\documentclass[11pt]{article}
\usepackage[font=small,margin=10pt,labelfont={bf},labelsep={space}]{caption}
\RequirePackage[numbers,sort&compress]{natbib}
\RequirePackage[colorlinks,citecolor=blue,urlcolor=blue, linkcolor=blue]{hyperref}
\usepackage{lineno,hyperref}
\usepackage{amsmath}
\usepackage{amssymb}
\usepackage{amsfonts}
\usepackage{bm}
\usepackage{bbm}
\usepackage{mathabx}
\usepackage[utf8]{inputenc}
\usepackage{csquotes}
\usepackage{amsthm}
\usepackage{dsfont}
\usepackage{leftidx}
\usepackage{enumerate}
\usepackage[OT4]{fontenc}
\usepackage{graphicx}
\usepackage{enumitem}
\usepackage{multirow}
\usepackage{rotating}
\usepackage{float}
\usepackage{booktabs}
\usepackage[ruled]{algorithm2e}
\def\ud{\, \mathrm{d}}
\usepackage{fancyhdr}
\usepackage{authblk}

\newtheorem{thm}{Theorem}[section]
\newtheorem{lem}[thm]{Lemma}

\newtheorem{prop}[thm]{Property}
\newtheorem{definition}[thm]{Definition}
\newtheorem{remark}[thm]{Remark}
\newtheorem{cor}[thm]{Corollary}

\numberwithin{thm}{section}
\numberwithin{equation}{section}

\begin{document}

\title{Series Representation of Jointly S$\alpha$S Distribution via Symmetric Covariations}

\author[1]{Yujia Ding\thanks{yujia.ding@cgu.edu}}
\author[1]{Qidi Peng\thanks{qidi.peng@cgu.edu}}
\affil[1]{Institute of Mathematical Sciences, Claremont Graduate University, Claremont, CA 91711, USA}

\date{}

\maketitle

\begin{abstract}
We introduce the notion of symmetric covariation, which is a new measure of dependence between two components of a symmetric $\alpha$-stable random vector, where the stability parameter $\alpha$ measures the heavy-tailedness of its distribution. Unlike  covariation that exists only when  $\alpha\in(1,2]$, symmetric covariation is well defined for all $\alpha\in(0,2]$. We show that  symmetric covariation can be defined using the proposed generalized fractional derivative, which has broader usages than those involved in this work. Several properties of symmetric covariation have been derived. These are either similar to or more general than those of the covariance functions in the Gaussian case. The main contribution of this framework is the representation of the characteristic function of bivariate symmetric $\alpha$-stable distribution via convergent series based on a sequence of symmetric covariations. This  series representation extends the one of bivariate Gaussian.
\begin{flushleft}

\end{flushleft}
\end{abstract}

\section{Background and Motivation}
\label{Background and Motivation}
 Stable distributions are a general family of probability distributions, which include many well-known distributions, such as Gaussian, Cauchy, and L\'evy  distributions. Recently, they have been successfully applied in various fields, such as network traffic modeling, economics, finance, physics, biology, geology, and signal processing \cite{Samorodnitsky1993linear, Taqqu1994,Stoev2002,Stoev2004,Xu2011modeling,Nolan2003stable}. Unlike Gaussian distributions, stable non-Gaussian distributions possibly have infinite variances or even infinite means; thus, they can be used for modeling many financial and physical phenomena that exhibit heavy-tailed behavior. The definition of stable distribution is given below (see also \cite{Taqqu1994,Nolan2003stable}).
 \begin{definition}
     Let $X_1$ and $X_2$ be independent copies of a random variable $X$. Then $X$ is said to be stable if for any $a,b > 0$, there exist $c>0$ and $e\in\mathbb R$ such that
     \begin{equation}
     \label{def:stable_distr}
     aX_1 + bX_2\stackrel{d}{=}cX + e,
     \end{equation}
     where $\stackrel{d}{=}$ denotes equality in distribution. $X$ is called strictly stable if (\ref{def:stable_distr}) holds with $e = 0$.
 \end{definition}

Stable distribution was first characterized by L\'evy \cite{Levy1925} as the only possible limiting law of sums of independent identically distributed random variables.  It was later popularized by Mandelbrot \cite{Mandelbrot1963} while describing the distributions of income and speculative prices, where the variability of these financial data is abnormally high. Since then, stable distributions have been proposed as a model for many types of physical, economic, and financial systems, when non-Gaussianity has been detected \cite{Byczkowski1993,Nolan1998,Nolan2001}.   In the area of theoretical probability, stable distributions are considered as a paradigmatic example of non-Gaussian distributions, for which many crucial functional-analysis properties, such as non-ergodicity, non-integrability and infinite divisibility, have been derived \cite{Nolan2003stable}. Although stable distributions have many intriguing mathematical properties, owing to the lack of an explicit form of probability distribution functions for all but s few specific cases (Gaussian, Cauchy, and L\'evy), the most concrete way to describe them is through their characteristic functions \cite{Taqqu1994}. 

In this paper, we focus on the so-called jointly symmetric $\alpha$-stable (S$\alpha$S) random vectors and study the measures of dependence of their marginal variables. As a natural extension of the zero-mean Gaussian random vector, the S$\alpha$S random vector is defined by its characteristic function as follows:
\begin{definition}
\label{def:stable_vec}
A real-valued random vector $\bm X=(X_1,\ldots,X_d)$ ($d\ge1$) is said to be jointly symmetric $\alpha$-stable (S$\alpha$S) with $\alpha\in(0,2]$ if and only if there exists some symmetric finite nonnegative measure $\bm{\varGamma_X}$ on the unit sphere $S_d=\{\bm s\in\mathbb R^d:\|\bm s\|=1\}$ ($\|\cdot\|$ denotes the Euclidean norm) such that 
\begin{equation}
\label{ch:stable}
\mathbb E\big(e^{i\langle\bm \theta, \bm X\rangle}\big)=\exp\left(-\int_{S_d}\left|\langle\bm \theta,\bm s\rangle\right|^\alpha\bm{\varGamma_X}(\ud \bm s)\right),~\mbox{for all $\bm\theta\in\mathbb R^d$},
\end{equation}
where:
\begin{itemize}
\item $\langle\bm \theta,\bm s\rangle=\sum_{k=1}^d\theta_ks_k$ denotes the canonical inner product of $\bm \theta$ and $\bm s$. 
\item $\alpha$ is called the stability parameter or characteristic exponent. It is one of the key parameters characterizing a jointly S$\alpha$S distribution, which measures the level of heavy tails of $\bm X$. In particular, when $\alpha\in(0,2)$, any $k$-th order moment of a stable random variable explodes once $k\ge\alpha$.
\item The symmetric finite nonnegative measure $\bm{\varGamma_X}$ is called the spectral measure of $\bm X$. This functional parameter, together with the stability parameter $\alpha$, fully captures the joint distribution of $\bm X$. It is worth noting that $\bm{\varGamma_X}$ is unique  in the non-Gaussian case, i.e., $\alpha\in(0,2)$; it is not unique when $\bm X$ is Gaussian ($\alpha=2$, see, e.g., \cite{Taqqu1994}).
\end{itemize}
\end{definition}

In the univariate case ($d=1$), $\bm{\varGamma_X}$ is degenerate to a nonnegative-valued parameter. Then for each single-valued S$\alpha$S random variable $X$, there is some $\gamma\ge0$ such that the characteristic function of $X$ takes the following particularly simple form:
$$
\mathbb E(e^{i\theta X})=\exp\left(-\gamma^\alpha|\theta|^\alpha\right),~\mbox{for $\theta\in\mathbb R$}.
$$
We then use the notation $S(\alpha,\gamma)$ to denote the distribution of $X$.

In the more general case ($d\ge2$), if we define the essential component of the characteristic function in (\ref{ch:stable})  by
\begin{equation}
\label{scaling}
\sigma_{\bm X}(\bm \theta)=\left(\int_{S_d}\left|\langle\bm \theta,\bm s\rangle\right|^\alpha\bm{\varGamma_X}(\ud \bm s)\right)^{1/\alpha},
\end{equation}
we can then rewrite (\ref{ch:stable}) as
$$
\mathbb E\big(e^{i\langle\bm \theta, \bm X\rangle}\big)=\exp\left(-\sigma_{\bm X}^\alpha(\bm \theta)\right),~\mbox{for all $\bm\theta\in\mathbb R^d$}.
$$
Here $\sigma_{\bm X}(\bm \theta)$ is called the scale parameter of the random variable $\langle\bm\theta,\bm X\rangle$.

As a family member of stable distributions, jointly S$\alpha$S random vectors are particularly useful for modeling high-frequency data that exhibit outliers with a large probability. For example, in  network traffic modeling, if traffic is generated by many users sharing a single fast link \cite{Willinger1997self}, and if each user has similar behavior and limited access to the link bandwidth, then a Gaussian process can be applied to approximate the physical on-off type model. However, if users have very irregular behavior patterns and unlimited access to the link bandwidth, then large peaks of traffic activity will be observed. These peaks in the traffic cannot be modeled by Gaussian or finite-variance processes as their behavior is no longer of the Poisson type. Another example is for asset price modeling wherein empirical studies show the so-called asymmetric leptokurtic features \cite{Nolan2003modeling}.  That is, the stock return distribution has skewness and a higher peak and has two heavier tails than those of the Gaussian distribution. Heavy-tailed processes, with possibly infinite variance distributions, are more natural and appropriate models in the above two examples. Therefore, with regard to the examples described above, jointly S$\alpha$S random vectors with $\alpha<2$  can be used to model the traffic of modern computer telecommunication networks \cite{Stoev2004}, and the stock returns \cite{Xu2011modeling}.  From a statistical point of view, most of the attention has been on estimating of the stability parameter $\alpha$ of jointly stable distributions \cite{Mcculloch1986simple,Davydov1999,Nolan2003modeling,Hu2009least,Xu2011modeling}. For a more general setting with stable processes (e.g., linear fractional and linear multifractional stable processes), the estimation of $\alpha$ and the process paths' H\"older parameter have both been extensively studied in \cite{Ayache2012,Ayache2015,Ayache2017}. In our framework, we study the dependence among the components of a jointly S$\alpha$S random vector. Consequently, a series representation of the bivariate S$\alpha$S random vector's characteristic function is derived.

A series representation of a characteristic function is instrumental in both probability theory and statistics as its corresponding probability distribution can be shown to be completely described using a sequence of coefficients (distribution parameters).  The characteristic function of a jointly S$\alpha$S random vector is real-valued,  and therefore its convergent series representation can serve as an approximation tool of the jointly S$\alpha$S distribution. The series representation can also help in the study of the dependence between the marginal distributions of a jointly S$\alpha$S random vector.

Recall that the characteristic function of a zero-mean bivariate Gaussian vector $\bm X=(X_1,X_2)$ can be represented via a convergent series in the following way:
\begin{equation}
\label{ch:gaussian}
\mathbb E\big(e^{i\langle\bm\theta,\bm X\rangle}\big)=e^{-\sigma_{\bm X}^\alpha(\bm\theta)}=\exp\left(-\sum_{k=0}^{+\infty}\frac{c_k(X_1,X_2,\alpha)}{k!}\theta_1^k\theta_2^{\alpha-k}\right),~\mbox{for all $\bm\theta\in\mathbb R^2$},
\end{equation}
where $\alpha=2$ and the sequence of real numbers $(c_k(X_1,X_2,\alpha))_{k\ge0}$ is given by
\begin{eqnarray*}
&&c_0(X_1,X_2,\alpha)=\frac{1}{2}Var(X_2),~c_1(X_1,X_2,\alpha)=Cov(X_1,X_2),\\
&&c_2(X_1,X_2,\alpha)=Var(X_1),~ \mbox{and}~c_k(X_1,X_2,\alpha)=0~\mbox{for}~k\ge3. 
\end{eqnarray*}
In other words, the probability distribution of a zero-mean bivariate Gaussian random vector is entirely captured by its second-order moments $Var$ $(X_1)$, $Var(X_2)$ and $Cov(X_1,X_2)$. Evidently, (\ref{ch:gaussian}) does not hold for all jointly S$\alpha$S distributions since their second-order moments are infinite when $\alpha<2$. Therefore, it is natural to wonder whether some series representation are similar to that given in (\ref{ch:gaussian}) exists for all bivariate S$\alpha$S random vectors. More specifically, 
for any bivariate S$\alpha$S random vector $\bm X=(X_1,X_2)$ with $\alpha\in(0,2]$, does there exist a sequence of real numbers $(\widetilde{c_k}(X_1,X_2,\theta_1,\theta_2,$ $\alpha))_{k\ge0}$ such that
\begin{equation}
\label{problem}
\sigma_{\bm X}^\alpha(\theta_1,\theta_2)=\sum_{k=0}^{+\infty}\frac{\widetilde{c_k}(X_1,X_2,\theta_1,\theta_2,\alpha)}{k!},~\mbox{for all $(\theta_1,\theta_2)\in\mathbb R^2$}?
\end{equation}
Comparing (\ref{ch:gaussian}) with (\ref{problem}) allows us to quickly see that if such representation in (\ref{problem}) exists for the jointly S$\alpha$S random vector with any $\alpha\in(0,2]$, the sequence $(\widetilde{c_k}(X_1,X_2,\theta_1,\theta_2,\alpha))_{k\ge0}$   in fact generalizes the joint moments of the Gaussian case ($\alpha=2$) to $\alpha\in(0,2]$. As a result, the sequence can then be used to represent the characteristic function of the bivariate S$\alpha$S random vector.

In S$\alpha$S distributions, there exist several extensions, or replacements, of the covariance functions in the  setting $\alpha\in(0,2)$. For example, Press \cite{Press1972} provided an explicit algebraic representation of the characteristic function of the multivariate stable distribution, and proposed an extended notion of the spectral correlation coefficient, that is applicable to a family of multivariate S$\alpha$S distributions. Paulauskas \cite{Paulauskas1976} elaborated on the properties of the characteristic function of multivariate stable distribution in \cite{Press1972} and generalized the spectral correlation coefficient. Further, Kanter \cite{Kanter1974} derived certain \enquote{linear dependence} between stable variables, and they showed that under some conditions, the conditional expectation of a stable variable given another one is linear. Subsequently, a new dependence measure of jointly S$\alpha$S random variables, called covariation, was proposed in the studies of Cambanis and Miller \cite{Cambanis1981} and Miller \cite{Miller1978}.  Covariation extends the covariance function when $\alpha=2$. However, as a dependence measure, covariation is not symmetric, does not apply for $\alpha\in(0,1]$, and as we understand it, does not satisfy (\ref{problem}).

A number of efforts have been made to overcome the asymmetry of covariation and address all $\alpha\in(0,2]$. For example,  codifference \cite{Taqqu1994}, together with its generalizations \cite{Astrauskas1991,Kokoszka1994,Kokoszka1995}, and the signed symmetric covariation coefficient \cite{Garel2009a}, were introduced as symmetric measures of dependence. Kodia and Garel \cite{Kodia2014} showed that in the case of sub-Gaussian random vectors, signed symmetric covariation coefficient coincides with the generalized spectral correlation coefficient in \cite{Paulauskas1976}. More recently, Damarackas and Paulauskas \cite{Damarackas2014} demonstrated that the spectral covariance term can be used to measure the dependence between more general random variables in the domain of normal attraction of general stable vectors; examples include the linear (Ornstein-Uhlenbeck, log-fractional, linear fractional) stable processes. Damarackas and Paulauskas \cite{Damarackas2016,Damarackas2017} studied the properties of a very general family of measures of dependence (defined by Equations (12) - (15) in \cite{Damarackas2017}), which are symmetric and well-defined for $\alpha\in(0,2]$. The covariation, the spectral covariance, and the $\alpha$-spectral covariance \cite{Damarackas2017} are in fact members of this family. 

None of the above concepts will lead to the solution of Eq. (\ref{problem}). Hence, our framework  aims to obtain a new type of dependence measurement, called symmetric covariation, which extends the covariance functions and satisfies Eq. (\ref{problem}). Note that this symmetric covariation belongs to the general family of measures of dependence introduced in \cite{Damarackas2017}.

The introduction to symmetric covariation is inspired by covariation. Proposed by Cambanis and Miller \cite{Cambanis1981} and Miller \cite{Miller1978},  covariation is introduced to substitute the covariance function of the jointly S$\alpha$S random vector when $\alpha\in(1,2)$. It is a measure of the dependence between any two variables in the jointly S$\alpha$S random vector. Recall that the conventional covariation is defined as follows:
\begin{definition}
\label{def:covariation}
Let $\bm{X}=(X_1,X_2)$ be an S$\alpha $S random vector with $\alpha\in(1,2]$ and spectral measure $\bm{\varGamma_{X}}$.
 The covariation of $X_1$ on $X_2$ is the real number
\begin{equation*}
[X_1,X_2]_\alpha=\int_{S_2}s_1s_2^{\langle \alpha-1\rangle}\bm{\varGamma_X}(\ud \bm s),
\end{equation*}
where $S_2=\{\bm s\in\mathbb R^2:\|\bm s\|=1\}$ and for any $p\in\mathbb R$, the so-called signed power $a^{\langle p\rangle }$ is defined by
\begin{equation}
\label{sign_power}
a^{\langle p\rangle}=
|a|^p\textnormal{sign}(a)~\mbox{with}~\textnormal{sign}(a)=\left\{
\begin{array}{ll}
1&~\mbox{if $a>0$},\\
0&~\mbox{if $a=0$},\\
-1&~\mbox{if $a<0$}.
\end{array}\right.
\end{equation}
\end{definition}

It is easy to observe that, covariation generalizes the covariance function of the Gaussian random vector $(X_1,X_2)$ (see \cite{Taqqu1994}), in the sense that 
$
[X_1,X_2]_2=\frac{1}{2}Cov(X_1,X_2).
$
Although the covariation $[X_1,X_2]_{\alpha}$ seems to be a natural extension of the covariance $Cov(X_1,X_2)$, unfortunately, it lacks some of the desirable and strong properties of the covariance. For example, covariation is non-symmetric and linear only on its first argument. Moreover, it is defined only for $\alpha\in(1,2]$, and cannot be easily extended to the case $\alpha\le1$. For a more detailed introduction to covariation, see Chapter 2 of \cite{Taqqu1994}.

In order to find $\widetilde{c_k}$'s in (\ref{problem}) for the jointly S$\alpha$S distributions, in what follows we aim to introduce a new type of covariation, which is symmetric and well-defined for all $\alpha\in(0,2]$. Recall that the conventional covariation can be equivalently defined by taking the first-order derivative of the scale parameter (see, e.g., Definition 2.7.3 in \cite{Taqqu1994}). I.e.,
\begin{equation}
\label{ordinary_derivative}
[X_1,X_2]_{\alpha}=\frac{1}{\alpha}\frac{\partial \sigma_{\bm X}^{\alpha}(\theta_1,\theta_2)}{\partial \theta_1}\Big|_{\theta_1=0,\theta_2=1}.
\end{equation}
However, the right-hand side of the above equation becomes ill-defined when $\alpha<1$. To generalize the conventional covariations to the case $\alpha\in(0,2]$ using a similar approach, a discussion of fractional derivative will be included. To this end we will propose a new type of fractional derivative that generalizes the Riemann-Liouville fractional derivative. This new proposed fractional derivative is well-defined over the entire $\mathbb R$, which has its own contribution to real analysis.

The rest of this paper is organized as follows. In Sect. \ref{section:fractional_derivative}, as our first contribution, we provide an  extension of the Riemann-Liouville fractional derivative to define symmetric covariations. In Sect. \ref{section:symmetric_covariation}, we define symmetric covariations and discuss the properties of symmetric covariations. In Sect. \ref{section:series}, we provide the solution of Eq. (\ref{problem}) through a convergent series representation of $\sigma_{\bm X}^\alpha(\theta_1,\theta_2)$, based on the symmetric covariations. In Sect. \ref{dependency}, we study the relationship between the symmetric covariations and the dependence of S$\alpha$S variables. Finally, in Sect. \ref{conclusion}, we summarizes the entire  framework and discusses some future problems on this topic. Technical lemmas and proofs of statements are given in ``Appendix".

\section[Generalized Riemann–Liouville Fractional Derivative]{Generalized Riemann–Liouville Fractional\\ Derivative}
\label{section:fractional_derivative}
There exist various notions of fractional derivatives and fractional integrals; examples include Riemann–Liouville fractional integrals, Hadamard fractional integrals, Riemann–Liouville fractional derivatives, Caputo fractional derivatives, and the composition of the left and right (or the right and left) Riemann-Liouville fractional order integrals. We refer the readers to the following literature  \cite{Podlubny1999,Diethelm2010,Miller1993,Atangana2016,Herzallah2014,Malinowska2015,Malinowska_2,Ciesielskia2017} and the references therein.

Among the above notions of fractional derivatives, Riemann–Liouville fractional derivative is one of the most natural extensions of the ordinary derivative. However, because it consists of the left and the right Riemann–Liouville fractional derivatives, and neither of them is defined over $\mathbb{R}$, we can not apply it as an replacement of the ordinary derivative to (\ref{ordinary_derivative}). Therefore, we propose the following generalized Riemann–Liouville fractional derivative.
\begin{definition}
\label{derivative}
Let $f$ be a continuous function defined over $\mathbb R$, and let $x,a\in \mathbb R$, $\beta\in\mathbb R_+$, then the fractional derivative of $f$ of order $\beta$ on $x$ is defined by
\begin{equation}
\label{ge:derivative}
{}_a^m\text{D}_x^\beta f(x)=\frac{1}{\Gamma(n-\beta)}\frac{\ud^n}{\ud x^n}\int_a^xf(t)(x-t)^{n}|x-t|^{-\beta-1}\textnormal{sign}^{m+1}(x-t)\ud t,
\end{equation}
where $\Gamma(\cdot)$ denotes the gamma function, $m=0$ or $1$ and $n=\lfloor \beta\rfloor+1$, with $\lfloor\cdot\rfloor$ being the floor number.
\end{definition}

Note that in Definition \ref{derivative}, the choices of $x$ and $a$ can be arbitrary in $\mathbb R$, and in fact $m$ can be any nonnegative integer. According to (\ref{ge:derivative}), however,  for any nonnegative integer $m$ and any real number $x$, we have
$$
{}_a^m\text{D}_x^\beta f(x)=\left\{
\begin{array}{ll}
{}_a^0\text{D}_x^\beta f(x)&\mbox{if $m$ is even},\\
{}_a^1\text{D}_x^\beta f(x)&\mbox{if $m$ is odd},
\end{array}\right.$$
so it suffices to pick $m$ only from $\{0,1\}$. This new proposed fractional derivative generalizes the left and right Riemann-Liouville fractional derivatives to a single form; moreover, it extends the ordinary integer-order derivative.
\begin{remark}
\label{rmk:left_right}
${}_a^m\text{D}_x^\beta$ with $\beta\in\mathbb R_+\backslash\mathbb Z_+$ extends both the left and right Riemann-Liouville fractional derivatives in the following sense:
\begin{itemize}
\item If $x\ge a$, then ${}_a^m\text{D}_x^\beta$ can be rewritten as
$$
{}_a^m\text{D}_x^\beta f(x)
=\frac{1}{\Gamma(n-\beta)}\frac{\ud^n}{\ud x^n}\int_a^xf(t)(x-t)^{n-\beta-1}\ud t,
$$
which is the left Riemann-Liouville fractional derivative of $f$ on $x$.
\item If $x<a$, then ${}_a^m\text{D}_x^\beta$ can be rewritten as
$$
    {}_a^m\text{D}_x^\beta f(x)=\frac{(-1)^{m+n}}{\Gamma(n-\beta)}\frac{\ud^n}{\ud x^n}\int_x^af(t)(t-x)^{n-\beta-1}\ud t.
$$
We observe that in this case $(-1)^m {}_a^m\text{D}_x^\beta f(x)$ is the right Riemann-Liouville fractional derivative of $f$ on $x$.
\end{itemize}
In conclusion we can alternatively write
\begin{equation}
\label{def:derivative}
{}_a^m\text{D}_x^\beta f(x)=\mathds 1_{\{x\ge a\}}{}_a\widetilde {\text{D}}_x^\beta f(x)+(-1)^m\mathds 1_{\{x<a\}}{}_x\widetilde {\text{D}}_a^\beta f(x),
\end{equation}
where ${}_a\widetilde{\text{D}}_x^\beta$ and ${}_x\widetilde {\text{D}}_a^\beta$ denote the left and right Riemann-Liouville fractional derivatives, respectively, i.e.,
\begin{eqnarray*}
&&{}_a\widetilde {\text{D}}_x^\beta f(x)=\frac{1}{\Gamma(n-\beta)}\frac{\ud^n}{\ud x^n}\int_a^xf(t)(x-t)^{n-\beta-1}\ud t,\\
&&{}_x\widetilde {\text{D}}_a^\beta f(x)=\frac{(-1)^n}{\Gamma(n-\beta)}\frac{\ud^n}{\ud x^n}\int_x^af(t)(t-x)^{n-\beta-1}\ud t.
\end{eqnarray*}
\end{remark}

\begin{remark}
When $\beta\in\mathbb{Z}_+$, we have for $x\in\mathbb R$,
\begin{eqnarray}\label{integer_case}
{}_{a}^{m}\text{D}_x^\beta f(x)&=&\frac{\ud^n}{\ud x^n}\int_a^xf(t)(x-t)^{n}|x-t|^{-n}\textnormal{sign}^{m+1}(x-t)\ud t\nonumber\\
&=&\mathds 1_{\{x\ge a\}}\frac{\ud^n}{\ud x^n}\int_a^xf(t)\ud t+\mathds 1_{\{x< a\}}(-1)^{m+n+1}\frac{\ud^n}{\ud x^n}\int_a^xf(t)\ud t\nonumber\\
&=&\textnormal{sign}^{m+n+1}(x-a)\frac{\ud^\beta  f(x)}{\ud x^\beta},
\end{eqnarray}
where $n=\beta+1$. 
In particular, if $m=M(\beta)$ with $M(k)=(k \mod 2)$, 
\begin{eqnarray*}
{}_{a}^{m}\text{D}_x^\beta f(x)=\mathds{1}_{\{x\ne a\}}\frac{\ud^\beta  f(x)}{\ud x^\beta},
\end{eqnarray*}
which extends the ordinary $\beta$-th order derivative.
\end{remark}

The following lemma expresses an important result of applying the new defined fractional derivative; the result will be involved in the second definition of symmetric covariations introduced in Sect. \ref{section:symmetric_covariation}. The proof of this lemma can be found in ``Appendix" (see Sect. \ref{subsection:proof_of_poly}).
\begin{lem}
\label{poly}
For $p>-1$, $x,a\in\mathbb R$, $\beta\ge0$, and $m\in\{0,1\}$, 
\begin{equation}
\label{dev:poly}
{}_{a}^{m}\text{D}_x^\beta\left(|x-a|^{p}\right)=\frac{\Gamma(p+1)}{\Gamma(p-\beta+1)}|x-a|^{p-\beta}\textnormal{sign}^{m}(x-a).
\end{equation}
\end{lem}

\section{Symmetric Covariations}
\label{section:symmetric_covariation}
In this section, we introduce the new measure of dependence between two variables of jointly S$\alpha $S distribution; it is called symmetric covariations. We compare its properties with those of covariation and covariance. This new measurement turns out to be symmetric and well-defined on $\alpha\in(0,2]$, making it more general than covariation.
\begin{definition}
\label{covariation}
Let $\bm{X}=(X_1,X_2)$ be an S$\alpha $S random vector with some $\alpha\in(0,2]$ and some spectral measure $\bm{\varGamma_{X}}$. The symmetric covariations of $X_1$ and $X_2$ are defined by:
for $\beta\ge0$ and $m=0$ or $1$,
\begin{equation}
\label{representation}
[X_1,X_2]_{\alpha,\beta,m}=\int_{S_2} K_{\alpha,\beta,m}(s_1,s_2)\bm{\varGamma_X}(\ud \bm s),
\end{equation}
where the bivariate function $K_{\alpha,\beta,m}$ is given as: for all $s_1,s_2\in\mathbb R$,
\begin{eqnarray}
\label{K}
K_{\alpha,\beta,m}(s_1,s_2)\hspace{-0.1cm}&=&\hspace{-0.1cm} |s_1|^{\beta}|s_2|^{\alpha-\beta}\textnormal{sign}^{m}(s_1s_2)\mathds 1_{\{|s_1|\le |s_2|\}}\nonumber\\&&+|s_1|^{\alpha-\beta}|s_2|^{\beta}\textnormal{sign}^{m}(s_1s_2)\mathds 1_{\{|s_1|>|s_2|\}}.
\end{eqnarray}
\end{definition}

\begin{remark}[On the function $K_{\alpha,\beta,m}$] From (\ref{K}) we can write
\begin{eqnarray*}
&&K_{\alpha,\beta,m}(s_1,s_2)= |s_1|^{\beta}|s_2|^{\alpha-\beta}\textnormal{sign}^{m}(s_1s_2)\mathds 1_{\{|s_1|< |s_2|\}}\\
&&\hspace{1cm}+|s_1|^{\alpha}\textnormal{sign}^{m}(s_1^2)\mathds 1_{\{|s_1|= |s_2|\}}+|s_1|^{\alpha-\beta}|s_2|^{\beta}\textnormal{sign}^{m}(s_1s_2)\mathds 1_{\{|s_1|>|s_2|\}},
\end{eqnarray*}
which indicates that the bivariate function $K_{\alpha,\beta,m}$ is symmetric. Moreover, since $K_{\alpha,\beta,m}$ satisfies the equations (12) - (14) in \cite{Damarackas2017}, the symmetric covariation in (\ref{representation}) in fact belongs to the general family of dependence measurements introduced in  \cite{Damarackas2017}.
\end{remark}

\begin{remark}[On the parameter $m$]
\label{rmk3}
In view of Definition \ref{covariation} we have: for $\alpha\in(0,2]$ and $\beta\ge0$,
\begin{eqnarray*}
&&[X_1,X_2]_{\alpha,\beta,0}=\int_{\{(s_1,s_2)\in S_2:~|s_1|\le|s_2|\}} |s_1|^{\beta}|s_2|^{\alpha-\beta}\bm{\varGamma_X}(\ud \bm s)\\
&&\hspace{3cm}+\int_{\{(s_1,s_2)\in S_2:~|s_1|>|s_2|\}} |s_1|^{\alpha-\beta}|s_2|^{\beta}\bm{\varGamma_X}(\ud \bm s)
\end{eqnarray*}
and 
\begin{eqnarray*}
&&[X_1,X_2]_{\alpha,\beta,1}=\int_{\{(s_1,s_2)\in S_2:~|s_1|\le|s_2|\}} s_1^{\langle\beta\rangle}s_2^{\langle\alpha-\beta\rangle}\bm{\varGamma_X}(\ud \bm s)\\
&&\hspace{3cm}+\int_{\{(s_1,s_2)\in S_2:~|s_1|>|s_2|\}} s_1^{\langle\alpha-\beta\rangle}s_2^{\langle\beta\rangle}\bm{\varGamma_X}(\ud \bm s),
\end{eqnarray*}
where the signed power $a^{\langle \cdot\rangle }$ is defined in (\ref{sign_power}).
\end{remark}

\begin{remark}[On the parameter $\beta$]  From the definition of $K_{\alpha,\beta,m}$ in (\ref{K}), we can easily observe that the mapping $\beta\mapsto [X_1,X_2]_{\alpha,\beta,0}$ is decreasing.
\end{remark}

\begin{remark}
\label{remark:relation_convertional_covariation} 
The following relationship between the symmetric covariations and conventional covariations holds: 
	when $\alpha\in(1,2]$, 
    \begin{equation}
    \label{syc_co}
    [X_1,X_2]_{\alpha,1,1}+[X_1,X_2]_{\alpha,\alpha-1,1}=[X_1,X_2]_{\alpha}+[X_2,X_1]_{\alpha},
    \end{equation}
    where $[X_1,X_2]_{\alpha}$ and $[X_2,X_1]_{\alpha}$ are the conventional covariations defined in Definition $\ref{def:covariation}$. In fact, (\ref{syc_co}) is a straightforward consequence of Eq. (\ref{result_1}).
\end{remark}

\begin{remark} Symmetric covariations generalize covariances, because of the fact that when $\bm X=(X_1, X_2)$ is a Gaussian vector, $2[X_1,X_2]_{2,1,1}$ is equal to the covariance $Cov(X_1,X_2)$. Indeed, taking $\alpha=2$, $\beta=1$ and $m=1$  in Definition $\ref{covariation}$, we have
$$[X_1,X_2]_{2,1,1}=\int_{ S_2} s_1^{\langle1\rangle}s_2^{\langle1\rangle}\bm{\varGamma_X}(\ud \bm s)=\int_{ S_2} s_1s_2\bm{\varGamma_X}(\ud \bm s),$$
where $\int_{ S_2} s_1s_2\bm{\varGamma_X}(\ud \bm s)$ can be shown to equal to $2^{-1}Cov(X_1,X_2)$ in the Gaussian case, according to Example 2.7.2 in \cite{Taqqu1994}.
\end{remark}

The theorem below states a definition of symmetric covariation for computational purpose. It is equivalent to Definition \ref{covariation}, but is not related to the representations in (\ref{representation}); it expresses the symmetric covariation through its characteristic function. This second definition also indicates that symmetric covariations are well-defined for all $\alpha\in(0,2]$.
\begin{thm}
\label{covariation_2}
Let $\bm{X}=(X_1,X_2)$ be an S$\alpha $S random vector with $\alpha\in(0,2]$ and spectral measure $\bm{\varGamma_{X}}$. The symmetric covariation of $X_1$ and $X_2$ can be equivalently defined by: for $\beta\ge0$ and $m=0$ or $1$,
\begin{eqnarray*}
&&[X_1,X_2]_{\alpha,\beta,m}=\frac{\Gamma(\alpha-\beta+1)}{\Gamma(\alpha+1)}\\
&&\hspace{0.2cm}\times\left(\lim_{(\theta_1,\theta_2)\to(0,1)}\int_{\{(s_1,s_2)\in S_2:~|s_1|\le|s_2|\}}\leftidx{_{-\theta_2s_2/s_1}^{~~~~~~~m}}{\text{D}}{_{\theta_1}^\beta}|\theta_1 s_1+\theta_2s_2|^{\alpha}\bm{\varGamma_X}(\ud \bm s)\right.\\
&&\hspace{0.2cm}\left.+\lim_{(\theta_1,\theta_2)\to(1,0)}\int_{\{(s_1,s_2)\in S_2:~|s_1|>|s_2|\}}\leftidx{_{-\theta_1s_1/s_2}^{~~~~~~~m}}{\text{D}}{_{\theta_2}^\beta}|\theta_1 s_1+\theta_2s_2|^{\alpha}\bm{\varGamma_X}(\ud \bm s)\right).
\end{eqnarray*}
\end{thm}
\begin{proof} 
For $\beta\ge0$, applying Lemma \ref{poly}, we obtain
\begin{eqnarray}\label{second_def_proof_1}
&&\int_{\{(s_1,s_2)\in S_2:~|s_1|\le|s_2|\}}{}_{-\theta_2s_2/s_1}^{~~~~~~~m}\text{D}_{\theta_1}^\beta\left(|s_1|^{\alpha}\left|\theta_1 -\left(-\frac{\theta_2s_2}{s_1}\right)\right|^{\alpha}\right)\bm{\varGamma_X}(\ud \bm s)\nonumber\\
&&=\frac{\Gamma(\alpha+1)}{\Gamma(\alpha-\beta+1)}\int_{\{(s_1,s_2)\in S_2:~|s_1|\le|s_2|\}}|s_1|^{\alpha}\left|\theta_1+\frac{\theta_2s_2}{s_1}\right|^{\alpha-\beta}\nonumber\\&&\hspace{1cm}\times\textnormal{sign}^m\left(\theta_1+\frac{\theta_2s_2}{s_1}\right)      \bm{\varGamma_X}(\ud \bm s)
\end{eqnarray}
and
\begin{eqnarray}\label{second_def_proof_2}
&&\int_{\{(s_1,s_2)\in S_2:~|s_1|>|s_2|\}}{}_{-\theta_1s_1/s_2}^{~~~~~~~m}\text{D}_{\theta_2}^\beta\left(|s_2|^{\alpha}\left|\theta_2 -\left(-\frac{\theta_1s_1}{s_2}\right)\right|^{\alpha}\right)\bm{\varGamma_X}(\ud \bm s)\nonumber\\
&&=\frac{\Gamma(\alpha+1)}{\Gamma(\alpha-\beta+1)}\int_{\{(s_1,s_2)\in S_2:~|s_1|>|s_2|\}}|s_2|^{\alpha}\left|\theta_2+\frac{\theta_1s_1}{s_2}\right|^{\alpha-\beta}\nonumber\\&&\hspace{1cm}\times\textnormal{sign}^m\left(\theta_2+\frac{\theta_1s_1}{s_2}\right)      \bm{\varGamma_X}(\ud \bm s).
\end{eqnarray}
Observe that 
\begin{eqnarray*}
&&\lim_{(\theta_1,\theta_2)\to(0,1)}\int_{\{(s_1,s_2)\in S_2:~|s_1|\le|s_2|\}}|s_1|^{\alpha}\left|\theta_1+\frac{\theta_2s_2}{s_1}\right|^{\alpha-\beta}\bm{\varGamma_X}(\ud \bm s)\\
&&=\int_{\{(s_1,s_2)\in S_2:~|s_1|\le|s_2|\}}|s_1|^{\beta}|s_2|^{\alpha-\beta}\bm{\varGamma_X}(\ud \bm s)\\
&&\le \int_{\{(s_1,s_2)\in S_2:~|s_1|\le|s_2|\}}|s_2|^{\alpha}\bm{\varGamma_X}(\ud \bm s)\\
&&\le \int_{S_2}|s_2|^{\alpha}\bm{\varGamma_X}(\ud \bm s)<+\infty
\end{eqnarray*}
and similarly,
\begin{eqnarray*}
&&\lim_{(\theta_1,\theta_2)\to(1,0)}\int_{\{(s_1,s_2)\in S_2:~|s_1|>|s_2|\}}|s_2|^{\alpha}\left|\theta_2+\frac{\theta_1s_1}{s_2}\right|^{\alpha-\beta} \hspace{-0.3cm}   \bm{\varGamma_X}(\ud \bm s)\\&&\le \int_{S_2}|s_1|^{\alpha}\bm{\varGamma_X}(\ud \bm s)<+\infty.
\end{eqnarray*}
Theorem \ref{covariation_2} then follows from taking the limits $(\theta_1,\theta_2)$ $\to(0,1)$ and $(\theta_1,\theta_2)$ $\to(1,0)$ on both hand sides of $(\ref{second_def_proof_1})$ and $(\ref{second_def_proof_2})$  respectively.
\end{proof}

Note that, using a similar approach of that in Theorem \ref{covariation_2}, the symmetric covariation can be extended to measure the dependence between each pair of the variables of the high-dimensional S$\alpha$S random vectors. This notion hence generalizes the covariance matrix of a Gaussian vector in the case $\alpha=2$.
\begin{remark}
Let $\bm{X}=(X_1,\cdots,X_d)$, $d\ge 1$ be a $d$-dimensional S$\alpha $S random vector with $\alpha\in(0,2]$ and spectral measure $\bm{\varGamma_{X}}$. Let the scale parameter $\sigma_{\bm X}(\bm \theta)$ be defined as in (\ref{scaling}). The symmetric covariation matrix $[\bm{X},\bm{X}]_{\alpha,\beta,m}$ is then defined by 
\begin{equation*}
    [\bm{X},\bm{X}]_{\alpha,\beta,m}=\left[
    \begin{array}{cccc}
        &[X_1,X_1]_{\alpha,\beta,m} &\cdots& [X_1,X_d]_{\alpha,\beta,m}  \\
        &\vdots&\cdots&\vdots\\
        &[X_d,X_1]_{\alpha,\beta,m} &\cdots& [X_d,X_d]_{\alpha,\beta,m}  
    \end{array}\right],
\end{equation*}
where for  $i,j\in\{1,\ldots,d\}$, $\beta\ge0$ and $m=0$ or $1$,
\footnotesize
\begin{eqnarray*}
&&\hspace{-0.6cm}[X_i,X_j]_{\alpha,\beta,m}=\frac{\Gamma(\alpha-\beta+1)}{\Gamma(\alpha+1)}\\
&&\hspace{-0.6cm}\times\left(\lim_{\substack{(\theta_j,\theta_k)\to(1,0) \\ 1\le k\le d,~k\ne j}}\int_{\{(s_i,s_j)\in S_2:~|s_i|\le|s_j|\}}\leftidx{_{\hspace{-0.2cm}-\sum_{1\le k'\le d, k'\ne i}\theta_{k'}s_{k'}/s_i}^{~~~~~~~~~~~~~~~~~~~~m}}{\text{D}}{_{\theta_i}^\beta}\left|\theta_i s_i+\hspace{-0.3cm}\sum_{\substack{1\le k'\le d \\ k'\ne i}}\hspace{-0.3cm}\theta_{k'}s_{k'}\right|^{\alpha}\hspace{-0.2cm}\bm{\varGamma_X}(\ud \bm s)\right.\\
&&\hspace{-0.6cm}\left.+\lim_{\substack{(\theta_i, \theta_k)\to(1,0) \\ 1\le k\le d,~k\ne i}}\int_{\{(s_i,s_j)\in S_2:~|s_i|>|s_j|\}}\leftidx{_{\hspace{-0.2cm}-\sum_{1\le k'\le d, k'\ne j}\theta_{k'}s_{k'}/s_j}^{~~~~~~~~~~~~~~~~~~~~m}}{\text{D}}{_{\theta_j}^\beta}\left|\theta_j s_j+\hspace{-0.3cm}\sum_{\substack{1\le k'\le d \\ k'\ne j}}\hspace{-0.3cm}\theta_{k'}s_{k'}\right|^{\alpha}\hspace{-0.2cm}\bm{\varGamma_X}(\ud \bm s)\right).
\end{eqnarray*}
\normalsize
\end{remark}

Let $\bm X=(X_1,\ldots,X_n)$ be an $n$-dimensional S$\alpha$S random vector, we study the symmetric covariation between the two weighted sums $Y_1=\sum_{k=1}^na_kX_k$ and $Y_2=\sum_{k=1}^nb_kX_k$. The main result is given in Theorem \ref{high_cov}, which requires the results of the following Lemma \ref{measure} and Lemma \ref{96} in its proof. The proof of Lemma \ref{measure} and Lemma \ref{96} can be found in Theorem 3.6.1 of \cite{Bogachev2007} and ``Appendix" (see Section \ref{subsection:proof_of_96}), respectively.
\begin{lem}\label{measure}
Given two spaces $X$ and $Y$ with $\sigma$-algebras $\mathcal{A}$ and $\mathcal{B}$ and an
$(\mathcal{A}, \mathcal{B})$-measurable mapping $f : X \rightarrow Y$. Then for any bounded (or bounded
from below) measure $\mu$ on $\mathcal{A}$, the formula
\begin{equation*}
\mu\circ f^{-1}:B\mapsto\mu\big(f^{-1}(B)\big),\quad B\in\mathcal{B}
\end{equation*}
defines a measure on $\mathcal{B}$ called the image of the measure $\mu$ under the mapping $f$.

If $\mu$ is a nonnegative measure, a $\mathcal{B}$-measurable function $g$ on $Y$ is integrable with respect to the measure $\mu\circ f^{-1}$, precisely when the function $g\circ f$ is integrable with respect to $\mu$. In addition, we have
\begin{equation*}
\int_{Y}g(y)\mu\circ f^{-1}(\ud y)=\int_{X}g(f(x))\mu(\ud x).
\end{equation*}
\end{lem}

\begin{lem}\label{96}
Let $\bm{X}=(X_1,\ldots,X_n)$, $n\ge2$ be an S$\alpha $S random vector with $\alpha\in(0,2]$ and spectral measure $\bm{\varGamma_{X}}$, and let $\bm{Y}=(Y_1,Y_2)$ with $Y_1=\sum_{k=1}^{n}a_kX_k$ and $Y_2=\sum_{k=1}^{n}b_kX_k$, where  $a_1,\ldots,a_n,b_1,\ldots,b_n\in\mathbb R$. Then $\bm{Y}$ is also an S$\alpha$S random vector with some spectral measure $\bm{\varGamma_{Y}}$ which can be expressed as
\begin{equation*}
\bm{\varGamma_{Y}}=\widehat{\bm{\varGamma_{X}}}(h^{-1})=\widehat{\bm{\varGamma_{X}}}\circ h^{-1},
\end{equation*}
where 
\footnotesize
\begin{eqnarray}\nonumber
&&h: S_n=\left\{(s_1,\cdots, s_n): \sum_{k=1}^{n}s_k^2=1\right\}\longrightarrow S_2=\left\{(t_1, t_2): \sum_{k=1}^{2}t_k^2=1\right\},\\\nonumber
&&h(s_1,\cdots, s_n)=\left(\frac{\sum\limits_{k=1}^{n}a_ks_k}{\sqrt{\big(\sum\limits_{k=1}^{n}a_ks_k\big)^{2}+\big(\sum\limits_{k=1}^{n}b_ks_k\big)^{2}}},\frac{\sum\limits_{k=1}^{n}b_ks_k}{\sqrt{\big(\sum\limits_{k=1}^{n}a_ks_k\big)^{2}+\big(\sum\limits_{k=1}^{n}b_ks_k\big)^{2}}}\right)\nonumber,
\end{eqnarray}
\normalsize
and
\begin{equation*}
\widehat{\bm{\varGamma_{X}}}(\ud\bm{s})=\left(\left(\sum_{k=1}^{n}a_ks_k\right)^{2}+\left(\sum_{k=1}^{n}b_ks_k\right)^{2}\right)^{\alpha/2}\bm{\varGamma_{X}}(\ud\bm{s}),\quad \bm{s}\in  S_n.
\end{equation*}
\end{lem}

\begin{thm}
	\label{high_cov}
	Let $\bm{X}=(X_1,\ldots,X_n)$ be an S$\alpha $S random vector with $\alpha\in(0,2]$ and spectral measure $\bm{\varGamma_{X}}$, and let $\bm{Y}=(Y_1,Y_2)$ be defined by $Y_1=\sum_{k=1}^{n}a_kX_k$ and $Y_2=\sum_{k=1}^{n}b_kX_k$. Then for $\beta\ge0$ and $m=0$ or $1$,
$$[Y_1,Y_2]_{\alpha,\beta,m}=\int_{S_n} K_{\alpha,\beta,m}\left(\sum_{k=1}^{n}a_ks_k,\sum_{k=1}^{n}b_ks_k\right)\bm{\varGamma_X}(\ud \bm s),
$$
where $K_{\alpha,\beta,m}$ is defined in (\ref{K}).
\end{thm}
\begin{proof}
By (\ref{K}) we observe: for $\alpha\in(0,2]$, $\beta\ge0$ and $m\in\{0,1\}$,
\begin{eqnarray}
\label{K1}
&&\int_{S_n} K_{\alpha,\beta,m}\left(\sum_{k=1}^{n}a_ks_k,\sum_{k=1}^{n}b_ks_k\right)\bm{\varGamma_X}(\ud \bm s)\nonumber\\
&&=\int_{\left\{\bm s\in S_n:~\left|\sum_{k=1}^{n}a_ks_k\right|\le\left|\sum_{k=1}^{n}b_ks_k\right|\right\}} \left|\sum_{k=1}^{n}a_ks_k\right|^{\beta}\left|\sum_{k=1}^{n}b_ks_k\right|^{\alpha-\beta}\nonumber\\
&&\hspace{2cm}\times\textnormal{sign}^{m}\left(\left(\sum_{k=1}^{n}a_ks_k\right)\left(\sum_{k=1}^{n}b_ks_k\right)\right)\bm{\varGamma_X}(\ud \bm s)\nonumber\\
&&+\int_{\left\{\bm s\in S_n:~\left|\sum_{k=1}^{n}a_ks_k\right|>\left|\sum_{k=1}^{n}b_ks_k\right|\right\}} \left|\sum_{k=1}^{n}a_ks_k\right|^{\alpha-\beta}\left|\sum_{k=1}^{n}b_ks_k\right|^{\beta}\nonumber\\
&&\hspace{2cm}\times\textnormal{sign}^{m}\left(\left(\sum_{k=1}^{n}a_ks_k\right)\left(\sum_{k=1}^{n}b_ks_k\right)\right)\bm{\varGamma_X}(\ud \bm s).
\end{eqnarray}
Then rewriting (\ref{K1}) and using the notations introduced in Lemma \ref{96}, we have
\small
\begin{eqnarray}
\label{K2}
&&\int_{S_n} K_{\alpha,\beta,m}\left(\sum_{k=1}^{n}a_ks_k,\sum_{k=1}^{n}b_ks_k\right)\bm{\varGamma_X}(\ud \bm s)\nonumber\\
&&=\int_{S_n}K_{\alpha,\beta,m}\left(\frac{\sum\limits_{k=1}^{n}a_ks_k}{\sqrt{\big(\sum\limits_{k=1}^{n}a_ks_k\big)^{2}+\big(\sum\limits_{k=1}^{n}b_ks_k\big)^{2}}},\frac{\sum\limits_{k=1}^{n}b_ks_k}{\sqrt{\big(\sum\limits_{k=1}^{n}a_ks_k\big)^{2}+\big(\sum\limits_{k=1}^{n}b_ks_k\big)^{2}}}\right)\nonumber\\
&&\hspace{4cm}\times\left(\Big(\sum_{k=1}^{n}a_ks_k\Big)^{2}+\Big(\sum_{k=1}^{n}b_ks_k\Big)^{2}\right)^{\alpha/2}\bm{\varGamma_X}(\ud \bm s)\nonumber\\
&&=\int_{S_n} K_{\alpha,\beta,m}\circ h(s_1,\ldots,s_n)\widehat{\bm{\varGamma_X}}(\ud \bm s).
\end{eqnarray}
\normalsize
Using the change of variable $(u_1,u_2)=h(s_1,\ldots,s_n)$ in (\ref{K2}) and applying Lemma \ref{96}, we obtain
\begin{eqnarray*}
&&\int_{S_n} K_{\alpha,\beta,m}\left(\sum_{k=1}^{n}a_ks_k,\sum_{k=1}^{n}b_ks_k\right)\bm{\varGamma_X}(\ud \bm s)\\&=&\int_{S_2} K_{\alpha,\beta,m}\left(u_1,u_2\right)\widehat{\bm{\varGamma_X}} (\ud h^{-1}(\bm u))\\
&=&[Y_1,Y_2]_{\alpha,\beta,m}.
\end{eqnarray*}
\end{proof}

In the following corollaries, we list a number of properties of symmetric covariations, which are mainly derived from the results in Theorem \ref{high_cov}. It turns out that Theorem \ref{high_cov} becomes one of the most important results in this paper, since it summarizes the key features of symmetric covariations.

\begin{cor}[Symmetry]
\label{symmetry}
Let $\bm{X}=(X_1,X_2)$ be an S$\alpha $S random vector with $\alpha\in(0,2]$ and spectral measure $\bm{\varGamma_{X}}$.
	The symmetric covariation is symmetric in its arguments, i.e. for $\beta\ge0$ and $m=0$ or $1$, 
$$
   [X_1,X_2]_{\alpha,\beta,m}=[X_2,X_1]_{\alpha,\beta,m}.
$$
\end{cor}
\begin{proof}
	This is a straightforward consequence of the facts that $K_{\alpha,\beta,m}$ is a symmetric bivariation function and ${\bm \Gamma_X}$ is a symmetric spectral measure.
\end{proof}

Unlike the conventional covariation, symmetric covariation shares the symmetric property of  covariance. Unfortunately, it lacks the linearity of covariance. However, it shows the scaling behavior. In particular, the behavior of the symmetric covariation of the scaled variables $aX_1$ and $bX_2$ is characterized by  Corollaries \ref{non-linear1} and  \ref{non-linear2}. 
\begin{cor}[Scaling]
\label{non-linear1}
Let $\bm{X}=(X_1,X_2)$ be an S$\alpha $S random vector with $\alpha\in(0,2]$ and spectral measure $\bm{\varGamma_{X}}$. Let $a$ and $b$ be two real numbers, then for $\beta=\alpha/2$ and $m=0$ or $1$,
\begin{equation*}
[aX_1,bX_2]_{\alpha,\alpha/2,m}=\left|a\right|^{\alpha/2}\left|b\right|^{\alpha/2}\textnormal{sign}^{m}(ab)[X_1,X_2]_{\alpha,\alpha/2,m}.
\end{equation*}
\end{cor}
\begin{proof}
Using (\ref{representation}) and Theorem $\ref{high_cov}$, we have
\footnotesize
\begin{eqnarray*}
&&[aX_1,bX_2]_{\alpha,\alpha/2,m}\\&&=\int_{\left\{(s_1, s_2)\in S_2:~\left|as_1\right|\le\left|bs_2\right|\right\}} \left|as_1\right|^{\alpha/2}\left|bs_2\right|^{\alpha/2}\textnormal{sign}^{m}(as_1bs_2)\bm{\varGamma_X}(\ud \bm s)\\
&&\hspace{1cm}+\int_{\left\{(s_1, s_2)\in S_2:~\left|as_1\right|>\left|bs_2\right|\right\}} \left|as_1\right|^{\alpha/2}\left|bs_2\right|^{\alpha/2}\textnormal{sign}^{m}(as_1bs_2)\bm{\varGamma_X}(\ud \bm s)\\
&&=\left|a\right|^{\alpha/2}\left|b\right|^{\alpha/2}\textnormal{sign}^{m}(ab)\left(\int_{\left\{(s_1, s_2)\in S_2:~\left|s_1\right|\le\left|s_2\right|\right\}} \left|s_1\right|^{\alpha/2}\left|s_2\right|^{\alpha/2}\textnormal{sign}^{m}(s_1s_2)\bm{\varGamma_X}(\ud \bm s)\right.\\
&&\hspace{1cm}\left.+\int_{\left\{(s_1, s_2)\in S_2:~\left|s_1\right|>\left|s_2\right|\right\}} \left|s_1\right|^{\alpha/2}\left|s_2\right|^{\alpha/2}\textnormal{sign}^{m}(s_1s_2)\bm{\varGamma_X}(\ud \bm s)\right)\\
&&=\left|a\right|^{\alpha/2}\left|b\right|^{\alpha/2}\textnormal{sign}^{m}(ab)[X_1,X_2]_{\alpha,\alpha/2,m}.
\end{eqnarray*}
\normalsize
\end{proof}

The preceding corollary expresses the mapping between the symmetric covariation of $aX_1$ and $bX_2$ and that of $X_1$ and $X_2$ when $\beta=\alpha/2$. It indicates that the symmetric covariation $[X_1,X_2]_{\alpha,\beta,m}$ becomes linear only when it is degenerate to the covariance function of the Gaussian, i.e., $\alpha=2$, $\beta=1$ and $m=1$. Corollary \ref{non-linear2} below shows that, in general, if one argument of the symmetric covariation is multiplied by a factor, then its value is multiplied by some power of this factor, but at the same time the scale factor of the other argument also changes.
\begin{cor}[Scaling]
\label{non-linear2}
Let $\bm{X}=(X_1,X_2)$ be an S$\alpha $S random vector with $\alpha\in(0,2]$ and spectral measure $\bm{\varGamma_{X}}$. Let $a$ and $b$ be two real numbers with $a\neq0$, then for $\beta\ge0$ and $m=0$ or $1$,
\begin{equation*}
[aX_1,bX_2]_{\alpha,\beta,m}=|a|^\alpha\left[X_1,\frac{b}{a}X_2\right]_{\alpha,\beta,m}.
\end{equation*}
\end{cor}
\begin{proof}
Using Theorem $\ref{high_cov}$ and the fact that $a\neq 0$, we have
\begin{eqnarray*}
&&[aX_1,bX_2]_{\alpha,\beta,m}\\&&=|a|^\alpha\int_{\left\{(s_1, s_2)\in S_2:~\left|s_1\right|\le\left|\frac{bs_2}{a}\right|\right\}} \left|s_1\right|^{\beta}\left|\frac{bs_2}{a}\right|^{\alpha-\beta}\textnormal{sign}^{m}(as_1bs_2)\bm{\varGamma_X}(\ud \bm s)\\
&&\hspace{1cm}+|a|^\alpha\int_{\left\{(s_1, s_2)\in S_2:~\left|s_1\right|>\left|\frac{bs_2}{a}\right|\right\}} \left|s_1\right|^{\alpha-\beta}\left|\frac{bs_2}{a}\right|^{\beta}\textnormal{sign}^{m}(as_1bs_2)\bm{\varGamma_X}(\ud \bm s)\\
&&=|a|^\alpha\left[X_1,\frac{b}{a}X_2\right]_{\alpha,\beta,m}.
\end{eqnarray*}
\end{proof}

The following result states that  symmetric covariation is an even function of each of its arguments, when $m=0$; it becomes an odd function of each of its arguments when $m=1$.
\begin{cor}[Scaled by sign]
Let $\bm{X}=(X_1,X_2)$ be an S$\alpha $S random vector with $\alpha\in(0,2]$ and spectral measure $\bm{\varGamma_{X}}$. For $\beta\ge0$ and $m=0$ or $1$,
\begin{equation*}
[X_1,-X_2]_{\alpha,\beta,m}=(-1)^m[X_1,X_2]_{\alpha,\beta,m}.
\end{equation*}
\end{cor}
\begin{proof}
Using Theorem $\ref{high_cov}$, we have
\begin{eqnarray*}
&&[X_1,-X_2]_{\alpha,\beta,m}\\&&=\int_{\left\{(s_1, s_2)\in S_2:~\left|s_1\right|\le\left|s_2\right|\right\}} \left|s_1\right|^{\beta}\left|s_2\right|^{\alpha-\beta}\textnormal{sign}^{m}(-s_1s_2)\bm{\varGamma_X}(\ud \bm s)\\
&&\hspace{2cm}+\int_{\left\{(s_1, s_2)\in S_2:~\left|s_1\right|>\left|s_2\right|\right\}} \left|s_1\right|^{\alpha-\beta}\left|s_2\right|^{\beta}\textnormal{sign}^{m}(-s_1s_2)\bm{\varGamma_X}(\ud \bm s)\\
&&=(-1)^m[X_1,X_2]_{\alpha,\beta,m}.
\end{eqnarray*}
\end{proof}

We are now in position to elaborate the exact relationship between  symmetric covariation and  covariation. The following proposition shows that the covariation  $[X_1,X_2]_\alpha$ can be written in terms of $[a_1X_1,b_1X_2]_{\alpha,1,1}$ and $[a_2X_1,b_2X_2]_{\alpha,\alpha-1,1}$ for some $a_1,b_1,a_2,b_2$. The proof is given in ``Appendix" (Section \ref{subsection:proof_of_prop}).
\begin{prop}
\label{prop:equiv} 
Assume $\alpha\in(1,2]$ and let $\bm{X}=(X_1,X_2)$ be an S$\alpha $S random vector.
\begin{description}
\item[(i)] For any $a,b\in\mathbb R$, 
 \begin{eqnarray}
    \label{result_1}
    &&[aX_1,bX_2]_{\alpha,1,1}+[aX_1,bX_2]_{\alpha,\alpha-1,1}\nonumber\\&&=ab^{\langle\alpha-1\rangle}[X_1,X_2]_{\alpha}+ba^{\langle\alpha-1\rangle}[X_2,X_1]_{\alpha}.
    \end{eqnarray}
\item[(ii)] For any $a_1,b_1,a_2,b_2\in\mathbb R$ such that $a_1b_1^{\langle\alpha-1\rangle}b_2a_2^{\langle\alpha-1\rangle}\ne a_2b_2^{\langle\alpha-1\rangle}b_1a_1^{\langle\alpha-1\rangle}$, 
\begin{equation}
\label{result_2}
\left\{\begin{array}{ll}
     &[X_1,X_2]_{\alpha}=\frac{c_1b_2a_2^{\langle\alpha-1\rangle}-c_2b_1a_1^{\langle\alpha-1\rangle}}{a_1b_1^{\langle\alpha-1\rangle}b_2a_2^{\langle\alpha-1\rangle}-a_2b_2^{\langle\alpha-1\rangle}b_1a_1^{\langle\alpha-1\rangle}}, \\
     &\\
     &[X_2,X_1]_{\alpha}=\frac{c_2a_1b_1^{\langle\alpha-1\rangle}-c_1a_2b_2^{\langle\alpha-1\rangle}}{a_1b_1^{\langle\alpha-1\rangle}b_2a_2^{\langle\alpha-1\rangle}-a_2b_2^{\langle\alpha-1\rangle}b_1a_1^{\langle\alpha-1\rangle}}, 
\end{array}\right.
\end{equation}
where for $j=1,2$, $c_j=[a_jX_1,b_jX_2]_{\alpha,1,1}+[a_jX_1,b_jX_2]_{\alpha,\alpha-1,1}$.
\end{description}
\end{prop}

Corollary \ref{high_low} below shows that the symmetric covariation of a single-valued S$\alpha$S random variable $X_1$ and itself is consistent with that of the conventional covariation.
\begin{cor}
\label{high_low}
Under the conditions of Theorem $\ref{high_cov}$, we have
\begin{equation}	\label{high_cov_sepc2}
[X_1,X_1]_{\alpha,\beta,m}=\int_{S_2}|s_1|^{\alpha}\bm{\varGamma_X}(\ud \bm s)=\sigma_{X_1}^{\alpha}(1),~\mbox{for all $\beta\ge0$ and $m=0,1$},
\end{equation}
where $\sigma_{X_1}$ is the scale parameter of the S$\alpha $S  random variable $X_1$. This result extends the conventional covariation $[X_1,X_1]_\alpha$.
\end{cor}

$(\ref{high_cov_sepc2})$ follows directly from Theorem $\ref{high_cov}$, by plugging in $a_1=b_1=1$ and  $a_i=b_j=0$ for $i>1, j>1$. The result of the preceding corollary indicates that $[X_1,X_1]_{\alpha,\beta,m}$ in fact does not depend on $\beta$ and $m$. Hence we define symmetric covariation norm as below. It implies that the symmetric covariation norm is equal to the scale parameter and consistent with the covariation norm associated with convariation.
\begin{definition}
	The symmetric covariation norm is defined as $\|X\|_{\alpha}=[X,X]_{\alpha,0,0}^{1/\alpha}=\sigma_X(1)$.
\end{definition}

Next we establish a fine inequality which extends the Cauchy-Schwarz inequality for the covariances in the Gaussian cases.
\begin{thm}
\label{thm:Holder}
	Let $\bm{X}=(X_1,X_2)$ be an S$\alpha $S random vector with $\alpha\in(0,2]$ and spectral measure $\bm{\varGamma_{X}}$. Let $\alpha\in(0,2]$ and $\beta\ge \alpha/2$. Then for $m=0$ or $1$,
\begin{equation}
\label{major_inequality}
\left|[X_1, X_2]_{\alpha,\beta,m}\right|\le\min\left\{\|X_1\|_\alpha^{\alpha\wedge\beta}\|X_2\|_\alpha^{\alpha-(\alpha\wedge\beta)},\|X_2\|_\alpha^{\alpha\wedge\beta}\|X_1\|_\alpha^{\alpha-(\alpha\wedge\beta)}\right\},
\end{equation}
where $\alpha\wedge\beta=\min\{\alpha,\beta\}$. Moreover, for $\beta\in[\alpha/2,\alpha]$, the equality holds if and only if $X_1=\lambda X_2$ with some $\lambda\in\mathbb{R}$.
\end{thm}
\begin{proof}
By Definition $\ref{covariation}$ and the triangle inequality, we can write
\begin{eqnarray}
\label{triangle1}
&&\left|[X_1, X_2]_{\alpha,\beta,m}\right|=\left|\int_{\{(s_1,s_2)\in S_2:~|s_1|\le| s_2|\}} |s_1|^{\beta}| s_2|^{\alpha-\beta}\textnormal{sign}^{m}(s_1 s_2)\bm{\varGamma_X}(\ud \bm s)\right.\nonumber\\
&&\hspace{2cm}\left.+\int_{\{(s_1,s_2)\in S_2:~|s_1|>|s_2|\}} |s_1|^{\alpha-\beta}| s_2|^{\beta}\textnormal{sign}^{m}(s_1 s_2)\bm{\varGamma_X}(\ud \bm s)\right|\nonumber\\
&&\le\int_{\{(s_1,s_2)\in S_2:~|s_1|\le| s_2|\}} |s_1|^{\beta}| s_2|^{\alpha-\beta}\bm{\varGamma_X}(\ud \bm s)\nonumber\\
&&\hspace{2cm}+\int_{\{(s_1,s_2)\in S_2:~|s_1|>| s_2|\}} |s_1|^{\alpha-\beta}| s_2|^{\beta}\bm{\varGamma_X}(\ud \bm s).
\end{eqnarray}
Because of the assumption $2\beta\ge\alpha$, we have $\alpha-\beta-(\alpha\wedge\beta)\le0$. Then, the following inequality holds.
\begin{eqnarray}
\label{ineq1}
&&\int_{\{(s_1,s_2)\in S_2:~|s_1|\le| s_2|\}} |s_1|^{\beta}| s_2|^{\alpha-\beta}\bm{\varGamma_X}(\ud \bm s)\nonumber\\
&&=\int_{\{(s_1,s_2)\in S_2:~|s_1|\le| s_2|\}} |s_1|^{\beta}|s_2|^{\alpha\wedge\beta}| s_2|^{\alpha-\beta-(\alpha\wedge\beta)}\bm{\varGamma_X}(\ud \bm s)\nonumber\\
&&\le\int_{\{(s_1,s_2)\in S_2:~|s_1|\le| s_2|\}} |s_1|^{\beta}|s_2|^{\alpha\wedge\beta}| s_1|^{\alpha-\beta-(\alpha\wedge\beta)}\bm{\varGamma_X}(\ud \bm s)\nonumber\\
&&=\int_{\{(s_1,s_2)\in S_2:~|s_1|\le| s_2|\}} |s_1|^{\alpha-(\alpha\wedge\beta)}|s_2|^{\alpha\wedge\beta}\bm{\varGamma_X}(\ud \bm s)<+\infty.
\end{eqnarray}
Note that the right-hand side integral in (\ref{ineq1}) is well defined due to the fact that $\alpha-(\alpha\wedge\beta)\ge0$. Next using the fact that $\beta-(\alpha\wedge\beta)\ge0$, we obtain
\begin{eqnarray}
\label{ineq2}
&&\int_{\{(s_1,s_2)\in S_2:~|s_1|>| s_2|\}} |s_1|^{\alpha-\beta}| s_2|^{\beta}\bm{\varGamma_X}(\ud \bm s)\nonumber\\
&&=\int_{\{(s_1,s_2)\in S_2:~|s_1|>| s_2|\}} |s_1|^{\alpha-(\alpha\wedge\beta)}|s_2|^{\alpha\wedge\beta}\left(\frac{| s_2|}{|s_1|}\right)^{\beta-(\alpha\wedge\beta)}\bm{\varGamma_X}(\ud \bm s)\nonumber\\
&&\le\int_{\{(s_1,s_2)\in S_2:~|s_1|>| s_2|\}} |s_1|^{\alpha-(\alpha\wedge\beta)}|s_2|^{\alpha\wedge\beta}\bm{\varGamma_X}(\ud \bm s)<+\infty.
\end{eqnarray}
Combining (\ref{triangle1}), (\ref{ineq1}) and (\ref{ineq2}), we obtain
\begin{eqnarray}
\label{bound_1}
&&\left|[X_1, X_2]_{\alpha,\beta,m}\right|\le\int_{\{(s_1,s_2)\in S_2:~|s_1|\le| s_2|\}} |s_1|^{\beta}| s_2|^{\alpha-\beta}\bm{\varGamma_X}(\ud \bm s)\nonumber\\
&&\hspace{2cm}+\int_{\{(s_1,s_2)\in S_2:~|s_1|>| s_2|\}} | s_1|^{\alpha-\beta}| s_2|^{\beta}\bm{\varGamma_X}(\ud \bm s)\nonumber\\
&&\le\int_{\{(s_1,s_2)\in S_2:~|s_1|\le| s_2|\}} |s_1|^{\alpha-(\alpha\wedge\beta)}|s_2|^{\alpha\wedge\beta}\bm{\varGamma_X}(\ud \bm s)\nonumber\\
&&\hspace{2cm}+\int_{\{(s_1,s_2)\in S_2:~|s_1|>| s_2|\}} |s_1|^{\alpha-(\alpha\wedge\beta)}|s_2|^{\alpha\wedge\beta}\bm{\varGamma_X}(\ud \bm s)\nonumber\\
&&=\int_{ S_2} |s_1|^{\alpha-(\alpha\wedge\beta)}|s_2|^{\alpha\wedge\beta}\bm{\varGamma_X}(\ud \bm s)<+\infty.
\end{eqnarray}
By the H\"older inequality, we can then bound the right-hand side of (\ref{bound_1}) by:
\begin{itemize}
\item If $\beta\ge\alpha$, 
\begin{equation}
\label{Holder1}
\int_{S_2} |s_1|^{\alpha-(\alpha\wedge\beta)}|s_2|^{\alpha\wedge\beta}\bm{\varGamma_X}(\ud s)=\int_{S_2} |s_2|^{\alpha}\bm{\varGamma_X}(\ud s)=\|X_2\|_{\alpha}^{\alpha}.
	\end{equation}
 \item If $\beta<\alpha$,
\begin{eqnarray}
\label{Holder2}
&&\int_{S_2} |s_1|^{\alpha-(\alpha\wedge\beta)}|s_2|^{\alpha\wedge\beta}\bm{\varGamma_X}(\ud s)=\int_{S_2} |s_1|^{\alpha-\beta}|s_2|^{\beta}\bm{\varGamma_X}(\ud s)\nonumber\\
&&\le\left(\int_{S_2} \left(|s_1|^{\alpha-\beta}\right)^{\frac{\alpha}{\alpha-\beta}}\bm{\varGamma_X}(\ud s)\right)^{1-\frac{\beta}{\alpha}}\left(\int_{S_2} \left(|s_2|^{\beta}\right)^{\frac{\alpha}{\beta}}\bm{\varGamma_X}(\ud s)\right)^{\frac{\beta}{\alpha}}\nonumber\\
 &&= \left(\int_{S_2} |s_1|^{\alpha}\bm{\varGamma_X}(\ud s)\right)^{1-\frac{\beta}{\alpha}}\left(\int_{S_2} |s_2|^{\alpha}\bm{\varGamma_X}(\ud s)\right)^{\frac{\beta}{\alpha}}\nonumber\\
 &&=\|X_1\|_{\alpha}^{\alpha-\beta}\|X_2\|_{\alpha}^{\beta}.
	\end{eqnarray}
    \end{itemize}
Hence, it follows from (\ref{bound_1}), (\ref{Holder1}) and (\ref{Holder2}) that
\begin{equation}
\label{inequality0}
\left|[X_1, X_2]_{\alpha,\beta,m}\right|\le\|X_1\|_{\alpha}^{\alpha-(\alpha\wedge\beta)}\|X_2\|_{\alpha}^{\alpha\wedge\beta}.
\end{equation}
Using the fact that $[X_1, X_2]_{\alpha,\beta,m}=[X_2, X_1]_{\alpha,\beta,m}$ and switching $X_1,X_2$ in (\ref{inequality0}), we obtain 
\begin{equation}
\label{inequality1}
\left|[X_1, X_2]_{\alpha,\beta,m}\right|=\left|[X_2, X_1]_{\alpha,\beta,m}\right|\le\|X_2\|_{\alpha}^{\alpha-(\alpha\wedge\beta)}\|X_1\|_{\alpha}^{\alpha\wedge\beta}.
\end{equation}
The inequality (\ref{major_inequality}) then follows from (\ref{inequality0}) and (\ref{inequality1}).

It remains to show that, when $\beta\in[\alpha/2,\alpha]$, the following equality holds if and only if $X_1=\lambda X_2$ for some $\lambda\in\mathbb R$. 
\begin{equation}
\label{equality}
\left|[X_1, X_2]_{\alpha,\beta,m}\right|=\min\left\{\|X_1\|_{\alpha}^{\alpha-\beta}\|X_2\|_{\alpha}^{\beta}, \|X_1\|_{\alpha}^{\beta}\|X_2\|_{\alpha}^{\alpha-\beta}\right\}.
\end{equation}

\begin{itemize}
\item If $X_1=\lambda X_2$ with $\lambda=0$, (\ref{equality}) holds obviously. If $X_1=\lambda X_2$ for some $\lambda\ne0$, then by the definition, we have that
\begin{eqnarray}
\label{equality1}
&&[\lambda X_2, X_2]_{\alpha,\beta,m}\nonumber\\&&=\int_{\{(s_1,s_2)\in S_2:~|\lambda s_2|\le| s_2|\}} |\lambda s_2|^{\beta}| s_2|^{\alpha-\beta}\textnormal{sign}^{m}(\lambda s_2 s_2)\bm{\varGamma_X}(\ud \bm s)\nonumber\\
&&\hspace{1cm}+\int_{\{(s_1,s_2)\in S_2:~|\lambda s_2|>| s_2|\}} |\lambda s_2|^{\alpha-\beta}| s_2|^{\beta}\textnormal{sign}^{m}(\lambda s_2 s_2)\bm{\varGamma_X}(\ud \bm s)\nonumber\\
&&=\int_{\{(s_1,s_2)\in S_2:~|\lambda|\le1\}} |s_2|^{\alpha}|\lambda|^{\beta}\textnormal{sign}^{m}(\lambda)\bm{\varGamma_X}(\ud \bm s)\nonumber\\
&&\hspace{1cm}+\int_{\{(s_1,s_2)\in S_2:~|\lambda|>1\}} |s_2|^{\alpha}|\lambda|^{\alpha-\beta}\textnormal{sign}^{m}(\lambda)\bm{\varGamma_X}(\ud \bm s)\nonumber\\
&&=\mathds 1_{\{|\lambda|\le1\}}|\lambda|^{\beta}\textnormal{sign}^{m}(\lambda)\int_{S_2} |s_2|^{\alpha}\bm{\varGamma_X}(\ud \bm s)\nonumber\\
&&\hspace{1cm}+\mathds 1_{\{|\lambda|>1\}}|\lambda|^{\alpha-\beta}\textnormal{sign}^{m}(\lambda)\int_{ S_2} |s_2|^{\alpha}\bm{\varGamma_X}(\ud \bm s)\nonumber\\
&&=\textnormal{sign}^{m}(\lambda)\left(\mathds 1_{\{|\lambda|\le1\}}|\lambda|^{\beta}+\mathds 1_{\{|\lambda|>1\}}|\lambda|^{\alpha-\beta}\right)\|X_2\|_{\alpha}^{\alpha}.
\end{eqnarray}
(\ref{equality1}) implies that the left-hand side of (\ref{equality}) equals
\begin{equation}
\label{equality1'}
\left|[\lambda X_2, X_2]_{\alpha,\beta,m}\right|=\left(\mathds 1_{\{|\lambda|\le1\}}|\lambda|^{\beta}+\mathds 1_{\{|\lambda|>1\}}|\lambda|^{\alpha-\beta}\right)\|X_2\|_{\alpha}^{\alpha},
\end{equation}
while the right-hand side of (\ref{equality}) equals
\begin{eqnarray}
\label{equality2}
&&\min\left\{\|\lambda X_2\|_{\alpha}^{\alpha-\beta}\|X_2\|_{\alpha}^{\beta}, \|\lambda X_2\|_{\alpha}^{\beta}\|X_2\|_{\alpha}^{\alpha-\beta}\right\}\nonumber\\
&&=\min\left\{|\lambda|^{\alpha-\beta}, |\lambda|^{\beta}\right\}\|X_2\|_{\alpha}^{\alpha}\nonumber\\
&&=\left(\mathds 1_{\{|\lambda|\le1\}}|\lambda|^{\beta}+\mathds 1_{\{|\lambda|>1\}}|\lambda|^{\alpha-\beta}\right)\|X_2\|_{\alpha}^{\alpha}.
\end{eqnarray}
It follows from (\ref{equality1'}) and (\ref{equality2}) that (\ref{equality}) holds for $X_1=\lambda X_2$, with $\lambda\neq0$. We conclude that (\ref{equality}) holds when $X_1=\lambda X_2$ for some $\lambda\in\mathbb{R}$.
\item Now suppose that (\ref{equality}) holds, we show $X_1=\lambda X_2$ for some $\lambda\in\mathbb R$. By (\ref{bound_1}) and (\ref{Holder2}) we have
\begin{eqnarray}
\label{equal_cond_3}
&&\left|[X_1,X_2]_{\alpha,\beta,m}\right|\nonumber\\&&=\min\left\{\int_{ S_2} |s_2|^{\beta}| s_1|^{\alpha-\beta}\bm{\varGamma_X}(\ud \bm s), \int_{ S_2} |s_1|^{\beta}| s_2|^{\alpha-\beta}\bm{\varGamma_X}(\ud \bm s)\right\}\nonumber\\
&&= \min\left\{\|X_2\|_{\alpha}^{\beta}\|X_1\|_{\alpha}^{\alpha-\beta}, \|X_1\|_{\alpha}^{\beta}\|X_2\|_{\alpha}^{\alpha-\beta}\right\}.
\end{eqnarray}
Now recall an elementary fact: for four real numbers $a_1,b_1,a_2,b_2$ such that $a_1\le a_2$, $b_1\le b_2$ and $\min\{a_1,b_1\}=\min\{a_2,b_2\}$, we necessarily have $a_1=a_2$ or $b_1=b_2$. Applying this fact to the following two H\"older inequalities (the H\"older inequality holds since $\beta\le\alpha$) 
\begin{eqnarray*}
&&\int_{ S_2} |s_2|^{\beta}| s_1|^{\alpha-\beta}\bm{\varGamma_X}(\ud \bm s)\le \|X_2\|_{\alpha}^{\beta}\|X_1\|_{\alpha}^{\alpha-\beta},\\
&&\int_{ S_2} |s_1|^{\beta}| s_2|^{\alpha-\beta}\bm{\varGamma_X}(\ud \bm s)\le \|X_1\|_{\alpha}^{\beta}\|X_2\|_{\alpha}^{\alpha-\beta},
\end{eqnarray*}
and the equation (\ref{equal_cond_3}), we should have either
$$
\int_{ S_2} |s_2|^{\beta}| s_1|^{\alpha-\beta}\bm{\varGamma_X}(\ud \bm s)= \|X_2\|_{\alpha}^{\beta}\|X_1\|_{\alpha}^{\alpha-\beta}
$$
or
$$
\int_{ S_2} |s_1|^{\beta}| s_2|^{\alpha-\beta}\bm{\varGamma_X}(\ud \bm s)=\|X_1\|_{\alpha}^{\beta}\|X_2\|_{\alpha}^{\alpha-\beta}.
$$
Observe that according to the H\"older inequality's boundary condition, either of the above two equations leads to $X_1=\lambda X_2$ for some $\lambda\in\mathbb R$.
\end{itemize}
Finally (\ref{equality}) holds for $\beta\in[\alpha/2,\alpha]$ if and only if $X_1=\lambda X_2$ for some $\lambda\in\mathbb R$ and Theorem \ref{thm:Holder} is proved.
\end{proof}

We summarize that, (\ref{major_inequality}) in fact consists of 2 inequalities:
\begin{itemize}
\item For $\beta\in[\alpha/2,\alpha]$,
\begin{equation}
\label{inequality_part1}
\left|[X_1, X_2]_{\alpha,\beta,m}\right|\le\min\left\{\|X_1\|_{\alpha}^{\alpha-\beta}\|X_2\|_{\alpha}^{\beta}, \|X_1\|_{\alpha}^{\beta}\|X_2\|_{\alpha}^{\alpha-\beta}\right\}.
\end{equation}
This inequality satisfies the H\"older inequality's boundary condition, so its equality holds if and only if $X_1=\lambda X_2$ for some $\lambda\in\mathbb R$.
\item For $\beta>\alpha$,
\begin{equation}
\label{inequality_part2}
\left|[X_1, X_2]_{\alpha,\beta,m}\right|\le\min\left\{\|X_1\|_{\alpha}^{\alpha}, \|X_2\|_{\alpha}^{\alpha}\right\}.
\end{equation}
It is easy to demonstrate that in this case
$$
\min\left\{\|X_1\|_{\alpha}^{\alpha}, \|X_2\|_{\alpha}^{\alpha}\right\}\ge\min\left\{\|X_1\|_{\alpha}^{\alpha-\beta}\|X_2\|_{\alpha}^{\beta}, \|X_1\|_{\alpha}^{\beta}\|X_2\|_{\alpha}^{\alpha-\beta}\right\}.
$$
In fact, the equality in (\ref{inequality_part2}) does not necessarily hold when $X_1=\lambda X_2$ for some $\lambda \in\mathbb R$. A contradiction can be given by assuming $X_1=2X_2\ne0$ and $\beta=3>\alpha$. In this case, by using (\ref{equality1'}) we have 
$$
\left|[X_1,X_2]_{\alpha,3,m}\right|=2^{\alpha-3}\|X_2\|_{\alpha}^{\alpha}\neq\min\left\{\|X_1\|_{\alpha}^{\alpha}, \|X_2\|_{\alpha}^{\alpha}\right\}=\|X_2\|_{\alpha}^{\alpha}.
$$
\end{itemize}

In the zero-mean Gaussian case when $\alpha=2$, $\beta=1$ and $m=1$, the inequality (\ref{inequality_part1}) is in fact the conventional Cauchy-Schwarz inequality for the covariance of the random variables $X_1$ and $X_2$:
$$
|Cov(X_1,X_2)|\le \sqrt{Var(X_1)Var(X_2)}.
$$

The inequality (\ref{inequality_part1}) thus inspires us to introduce a new type of \enquote{correlation coefficient} between $X_1$ and $X_2$: if we define, for $\beta\in[\alpha/2,\alpha]$ and $m=0,1$,
\begin{equation}
\label{correlation}
\rho_{\alpha,\beta,m}(X_1,X_2)=\frac{[X_1,X_2]_{\alpha,\beta,m}}{\min\left\{\|X_2\|_{\alpha}^\beta\|X_1\|_{\alpha}^{\alpha-\beta},\|X_1\|_{\alpha}^\beta\|X_2\|_{\alpha}^{\alpha-\beta}\right\}},
\end{equation}
then $\rho_{\alpha,\beta,m}(X_1,X_2)$ satisfies the following properties.
\begin{description}
\item[(1)] $\rho_{\alpha,\beta,m}(X_1,X_2)$ is symmetric:
$$
\rho_{\alpha,\beta,m}(X_1,X_2)=\rho_{\alpha,\beta,m}(X_2,X_1),
$$
for all jointly S$\alpha$S random vector $(X_1,X_2)$.
\item[(2)] $\rho_{\alpha,\beta,m}(X_1,X_2)\in[-1,1]$. Moreover, from (\ref{equality1}) and (\ref{equality2}) we can derive that, if $X_1=\lambda X_2$ for some $\lambda\in\mathbb R$,
$$
\rho_{\alpha,\beta,m}(X_1,X_2)=\textnormal{sign}^{m}(\lambda).
$$
\item[(3)] $\rho_{\alpha,\beta,m}(X_1,X_2)$ becomes the conventional correlation coefficient between Gaussian variables, when $\alpha=2$, $\beta=1$, $m=1$ and $X_1,X_2$ have zero mean.
\end{description}

In Sect. \ref{dependency}, we will show that the symmetric covariation $[X_1,X_2]_{\alpha,\beta,m}$, can be viewed as a measure of dependence between $X_1$ and $X_2$, which indicates that the new proposed \enquote{correlation coefficient}  $\rho_{\alpha,\beta,m}(X_1,X_2)$ is indeed another measure of bivariate dependence.

\section[Series Representation of the Characteristic Function of Bivariate S\texorpdfstring{$\alpha $}{TEXT}S  Distribution]{Series Representation of the Characteristic\\ Function of Bivariate S\texorpdfstring{$\alpha $}{TEXT}S  Distribution}
\label{section:series}
In this section, we answer the question raised in Sect. \ref{Background and Motivation} by a series representation of $\sigma^\alpha_{(X_1,X_2)}(\theta_1,\theta_2)$ in terms of the symmetric covariations. In particular, Theorem \ref{lem:scale} provides the explicit form of the  $(\widetilde{c_k}(X_1,X_2,\theta_1,\theta_2,\alpha))_{k\ge0}$s in terms of the symmetric covariations such that (\ref{problem}) holds. 
\subsection{Taylor Series}
We now focus on getting a series expansion of $|x+b|^\alpha$ with $b\neq 0$ and $\alpha>0$. It will be applied to derive the series representation of $\sigma^\alpha_{(X_1,X_2)}(\theta_1,\theta_2)$ of the jointly S$\alpha $S distribution in a later section. First recall the following classical result of Binomial series.
\begin{lem}[Binomial series]
\label{taylor}
Let $f(x)=(x+1)^\alpha$ with $x+1\ge0$ and $\alpha\ge0$. Then for $|x|\le 1$, $f(x)$ admits the following series expansion, which is absolutely convergent:
	\begin{equation*}
	f(x)=\sum_{k=0}^{+\infty}\frac{(\alpha)_k}{k!}x^k,
	\end{equation*}
    where $(\alpha)_k$ denotes the falling factorials, defined by: for nonnegative integer $k$,
$$
(\alpha)_k=\left\{
\begin{array}{ll}
\frac{\Gamma(\alpha+1)}{\Gamma(\alpha+1-k)}&~\mbox{if $\alpha\notin\mathbb Z_+$ or if $\alpha\in\mathbb Z_+$ and $k\le \alpha$},\\
0&~\mbox{otherwise}.
\end{array}\right.
$$
\end{lem}

Let $f(x)=|x+b|^\alpha$ with $b\neq 0$ and $\alpha>0$, applying the Binomial series, a convergent Taylor series representation of the particular function is provided via the following proposition. 
\begin{prop}
\label{rmk:taylor_abs}
Let $b\neq0,\alpha>0$. The mapping $x\mapsto|x+b|^{\alpha}$ has the following convergent series representation:
\begin{equation}
\label{taylor_abs}
|x+b|^\alpha=\sum_{k=0}^{+\infty}\frac{(a)_k}{k!}|b|^{\alpha-k}\textnormal{sign}^k(b)x^k,~\mbox{for all $x\in[-|b|,|b|]$}.
\end{equation}

\end{prop}
\begin{proof}
First observe that, for $b\ne0$ and $x\in[-|b|,|b|]$,
$$
|x+b|^{\alpha}=\left\{
\begin{array}{ll}
(x+b)^\alpha&\mbox{if $b>0$},\\
(-x-b)^\alpha&\mbox{if $b<0$}.
\end{array}\right.
$$
Next we consider the above two cases separately.
\begin{itemize}
\item If $b>0$, the mapping $x\mapsto(x+b)^{\alpha}$ is infinitely differentiable at point $0$ and $(x+b)^\alpha=b^\alpha(x/b+1)^\alpha$. Using Lemma \ref{taylor}, the convergent Taylor series for $x\mapsto(x+b)^{\alpha}$ around $0$ can be expressed as: for $x\in[-|b|,|b|]$,
$$
(x+b)^{\alpha}=\sum_{k=0}^{+\infty}\frac{(\alpha)_k}{k!}x^kb^{\alpha-k}.
$$
\item If $b<0$, similarly, we have $(-x-b)^\alpha=(-b)^\alpha(x/b+1)^\alpha$ and then the Taylor series for $x\mapsto(-x-b)^{\alpha}$ at point $0$ is given by: for $x\in[-|b|,|b|]$,
$$
(-x-b)^{\alpha}=\sum_{k=0}^{+\infty}\frac{(\alpha)_k(-1)^k}{k!}x^k(-b)^{\alpha-k}.
$$
\end{itemize}
Finally, (\ref{taylor_abs}) is obtained by combining the above two cases.
\end{proof}

\subsection{Series Representation via Symmetric Covariations}

Theorem  \ref{lem:scale} below provides a series representation of the scale parameter to the power of $\alpha$ via symmetric covariations.
\begin{thm}
\label{lem:scale}
Let $\bm{X}=(X_1,X_2)$ be an S$\alpha $S random vector with $\alpha\in(0,2]$ and spectral measure $\bm{\varGamma_{X}}$.
	The scale parameter of $\langle\bm\theta,\bm X\rangle$ to the power of $\alpha$  can be expressed via the following series representation: 
	\begin{equation}
    \label{sigma:taylor}
\sigma_{\bm X}^\alpha(\theta_1,\theta_2)=\sum_{k=0}^{+\infty}\frac{(\alpha)_k}{k!}[\theta_1X_1,\theta_2X_2]_{\alpha,k,M(k)},
	\end{equation}
	where $M(k)=(k \mod 2)$.
    \end{thm}
    
	\begin{proof}
		By (\ref{scaling}) we can write
\begin{equation}
\label{sigma}
\sigma_{\bm X}^\alpha(\theta_1,\theta_2)=	\int_{S_2}\Big|s_1\theta_1+\theta_2s_2\Big|^{\alpha}\bm{\varGamma_X}(\ud \bm s)=T_1(\theta_1,\theta_2)+T_2(\theta_1,\theta_2),
\end{equation}
where
\begin{equation}
        \label{T1}
        T_1(\theta_1,\theta_2)=\int_{\{(s_1,s_2)\in S_2:~|\theta_1s_1|\le|\theta_2s_2|\}}\Big|\theta_1s_1+\theta_2s_2\Big|^{\alpha}\bm{\varGamma_X}(\ud \bm s)
        \end{equation}
        and
        \begin{equation}
        \label{T2}
        T_2(\theta_1,\theta_2)=\int_{\{(s_1,s_2)\in S_2:~|\theta_1s_1|>|\theta_2s_2|\}}\Big|\theta_1s_1+\theta_2s_2\Big|^{\alpha}\bm{\varGamma_X}(\ud \bm s).
        \end{equation}
        For $|\theta_1s_1|\le |\theta_2s_2|$, plugging $x=\theta_1s_1$ and $b=\theta_2s_2$ in (\ref{taylor_abs}), we obtain
        \begin{eqnarray*}
        \label{T1:kernel_bound}
        &&|\theta_1s_1+\theta_2s_2|^\alpha=\sum_{k=0}^{+\infty}\frac{(\alpha)_k}{k!}|\theta_2s_2|^{\alpha-k}\textnormal{sign}^k(\theta_2s_2)(\theta_1s_1)^k\nonumber\\
        &&=\sum_{k=0}^{+\infty}\frac{(\alpha)_k}{k!}|\theta_1s_1|^k|\theta_2s_2|^{\alpha-k}\textnormal{sign}^k(\theta_1\theta_2s_1s_2).
        \end{eqnarray*}
        Observe that, when $|\theta_1s_1|\le |\theta_2s_2|$,
        \begin{eqnarray*}
        &&\sup_{N\ge0}\left|\sum_{k=0}^{N}\frac{(\alpha)_k}{k!}|\theta_1s_1|^k|\theta_2s_2|^{\alpha-k}\textnormal{sign}^k(\theta_1\theta_2s_1s_2)\right|\\
        &&\le \sum_{k=0}^{+\infty}\frac{|(\alpha)_k|}{k!}|\theta_2s_2|^{\alpha-k}|\theta_1s_1|^k\le |\theta_2s_2|^{\alpha}\sum_{k=0}^{+\infty}\frac{|(\alpha)_k|}{k!}<+\infty
        \end{eqnarray*}
    and   
  \begin{eqnarray*}
  &&\int_{\{(s_1,s_2)\in S_2:~|\theta_1s_1|\le|\theta_2s_2|\}}|\theta_2s_2|^{\alpha}\sum_{k=0}^{+\infty}\frac{|(\alpha)_k|}{k!}\bm{\varGamma_X}(\ud \bm s)\\
  &&\le |\theta_2|^\alpha[X_2,X_2]_{\alpha,0,0}\sum_{k=0}^{+\infty}\frac{|(\alpha)_k|}{k!}\nonumber<+\infty.
\end{eqnarray*}
    Then we can apply Lebesgue dominating convergence theorem to obtain
    \begin{eqnarray}
&&T_1(\theta_1,\theta_2)=\sum_{k=0}^{+\infty}\frac{(\alpha)_k}{k!}\int_{\{(s_1,s_2)\in S_2:~|\theta_1s_1|\le|\theta_2s_2|\}}\nonumber\\
&&\hspace{2cm}|\theta_1s_1|^k|\theta_2s_2|^{\alpha-k}\textnormal{sign}^k(\theta_1\theta_2s_1s_2)\bm{\varGamma_X}(\ud \bm s).\label{T1:series}
        \end{eqnarray}
        Similarly we can show that, for $|\theta_1s_1|>|\theta_2s_2|$,
        \begin{eqnarray}
        &&T_2(\theta_1,\theta_2)=\sum_{k=0}^{+\infty}\frac{(\alpha)_k}{k!}\int_{\{(s_1,s_2)\in S_2:~|\theta_1s_1|>|\theta_2s_2|\}}\nonumber\\
&&\hspace{2cm}|\theta_1s_1|^{\alpha-k}|\theta_2s_2|^k\textnormal{sign}^k(\theta_1\theta_2s_1s_2)\bm{\varGamma_X}(\ud \bm s).        \label{T2:series}
        \end{eqnarray}
        Then, (\ref{sigma:taylor}) follows from (\ref{sigma}) - (\ref{T2:series}) and Definition \ref{covariation}.
		\end{proof}
		
		According to Theorem \ref{lem:scale}, (\ref{problem}) holds by taking 
		$$
		\widetilde{c_k}(X_1,X_2,\theta_1,\theta_2, \alpha)=(\alpha)_k[\theta_1X_1,\theta_2X_2]_{\alpha,k,M(k)}.
		$$
		Further, from Theorem \ref{lem:scale} and Corollary \ref{non-linear2} we note that $(\theta_1,\theta_2)$ and $(X_1,X_2)$ are generally not separable in the series representation of the non-Gaussian case, because  symmetric covariations are generally not linear (see, e.g., Corollary \ref{non-linear2}). It is consistent with the \enquote{indivisibility} of  stable non-Gaussian distributions. In addition, it is easy to verify the impossibility of finding a sequence $(c_k(X_1,X_2,\alpha))_{k\ge0}$ with $\alpha\neq2$ that does not depend on $\theta_1$ and $\theta_2$ but satisfies
		$$
\sigma_{\bm X}^\alpha(\theta_1,\theta_2)=\sum_{k=0}^{+\infty}\frac{c_k(X_1,X_2,\alpha)\theta_1^k\theta_2^{\alpha-k}}{k!},~\mbox{for all $\theta_1,\theta_2\in\mathbb R$}.
$$
This above fact also indicates that the expression of $\widetilde{c_k}(X_1,X_2,\theta_1,\theta_2, \alpha)$ in Theorem \ref{lem:scale} is reasonably non-separable. However, the only exception is the Gaussian case when $\alpha=2$, which is discussed in details in the remark below.
\begin{remark}
Assume now $\alpha=2$. Using the facts that $(2)_k=0$ for all $k>2$ and $$
[X_1,X_2]_{2,0,0}+[X_1,X_2]_{2,2,0}=\|X_1\|_{2}^2+\|X_2\|_{2}^2,
$$
Theorem \ref{lem:scale} implies
\begin{eqnarray*}\nonumber
		&&\sigma_{\bm X}^2(\theta_1,\theta_2)=\sum_{k=0}^{2}\frac{(2)_k}{k!}[\theta_1X_1,\theta_2X_2]_{2,k,M(k)}\nonumber\\    
        &&=[\theta_1X_1,\theta_2X_2]_{2,0,0}+2[\theta_1X_1,\theta_2X_2]_{2,1,1}+[\theta_1X_1,\theta_2X_2]_{2,2,0}\\
&&=\theta_1^2\|X_1\|_{2}^2+2\theta_1\theta_2[X_1,X_2]_{2,1,1}+\theta_2^2\|X_2\|_{2}^2\\
        &&=\frac{1}{2}\theta_2^2\text{Var}(X_2)+\theta_1\theta_2\text{Cov}(X_1,X_2)+\frac{1}{2}\theta_1^2\text{Var}(X_1).
		\end{eqnarray*}
       The series representation now is consistent with that in the Gaussian case (see (\ref{ch:gaussian})). 
\end{remark}

\section{Symmetric Covariations and James Orthogonality}
\label{dependency}

The conventional covariation is tightly related to some sort of dependence between the coordinates of a jointly S$\alpha$S random vector. For example, the conventional covariation $[X_1,X_2]_{\alpha}$ for $\alpha\in(1,2)$ generally does not define a scalar product and the linear space of  S$\alpha$S random variables equipped with such covariation is not a pre-Hilbert space. Therefore, instead of using the classical notion of orthogonality which is only valid in (pre-)Hilbert spaces, James orthogonality is involved in the normed vector space endowed with the covariation norm \cite{Taqqu1994}.

\begin{definition}
Let $(E,\|\cdot\|)$ be a normed vector space. An element $x\in E$ is said to be James orthogonal to another element $y\in E$ (denoted by $x\perp_J y$) if 
$$
\|x+\lambda y\|\ge\|x\|,~\mbox{for any $\lambda\in\mathbb R$}.
$$
\end{definition}

James orthogonality naturally extends the classical notion of orthogonality, in the sense that if $E$ is a pre-Hilbert space equipped with the scalar product $\langle\cdot,\cdot\rangle$, then $\langle x,y\rangle=0$ implies $x\perp_Jy$ for any $x,y\in E$. It is shown (see, e.g., Proposition 2.9.2 in \cite{Taqqu1994}) that the zero covariation (i.e., $[X_1,X_2]_{\alpha}=0$) is equivalent to $X_1\perp_J X_2$. In this section, we will mainly discuss how to relate  symmetric covariations to independence and James orthogonality, respectively.

Similar to the covariance functions, symmetric covariations measure the dependence between two variables of jointly S$\alpha$S distribution. Below are two  necessary conditions, involving the symmetric covariations, for the two variables to be independent. 
The following Proposition \ref{indep1} states that independence implies vanishing symmetric covariation, which is a property similar to that of covariance.
\begin{prop}[Necessary condition for independence]
\label{indep1}
	Let $\bm{X}=(X_1,X_2)$ be an S$\alpha $S random vector with $\alpha\in(0,2]$ and spectral measure $\bm{\varGamma_{X}}$. If $X_1$ and $X_2$ are independent, then $\bm{\varGamma_{X}}$ is a discrete measure and for $\beta\ge0$, $m\in \{0,1\}$,
	\small
    \begin{equation*}
    [X_1,X_2]_{\alpha,\beta,m}=\left\{
    \begin{array}{ll}
    \bm{\varGamma_X}(\{(0,1),(0,-1)\})+\bm{\varGamma_X}(\{(1,0),(-1,0)\})&~\mbox{if $\beta=0$, $m=0$},\\
    0&~\mbox{otherwise.}
    \end{array}\right.
    \end{equation*} 
    \normalsize
\end{prop}
\begin{proof}
According to Property 2.7.11 in \cite{Taqqu1994}, the independence of $X_1$ and $X_2$ implies that the spectral measure $\bm{\varGamma_X}$ of $(X_1,X_2)$ must be concentrated on the points $(1,0), (-1,0),$ $(0,1)$, and $(0,-1)$ of the unit sphere $S_2$. By Definition \ref{covariation}, we then have, for $\beta\ge0$ and $m=0$ or $1$,
\begin{eqnarray*}
&&[X_1,X_2]_{\alpha,\beta,m}=\int_{\{(s_1,s_2)\in S_2:~|s_1|\le|s_2|\}} |s_1|^{\beta}|s_2|^{\alpha-\beta}\textnormal{sign}^{m}(s_1s_2)\bm{\varGamma_X}(\ud \bm s)\\
&&\hspace{3cm}+\int_{\{(s_1,s_2)\in S_2:~|s_1|>|s_2|\}} |s_1|^{\alpha-\beta}|s_2|^{\beta}\textnormal{sign}^{m}(s_1s_2)\bm{\varGamma_X}(\ud \bm s)\\&&\hspace{2cm}= |0|^{\beta}|1|^{\alpha-\beta}\textnormal{sign}^{m}(0)\bm{\varGamma_X}(\{(0,1),(0,-1)\})\\&&\hspace{3cm}+ |1|^{\alpha-\beta}|0|^{\beta}\textnormal{sign}^{m}(0)\bm{\varGamma_X}(\{(1,0),(-1,0)\}),
\end{eqnarray*}
because the support of $\bm{\varGamma_X}$ is such that either $s_1$ or $s_2$ is zero. Note that in the above equation, $[X_1,X_2]_{\alpha,\beta,m}$ is not necessarily vanishing only if $\beta=m=0$, i.e.,
$$
[X_1,X_2]_{\alpha,0,0}=\bm{\varGamma_X}(\{(0,1),(0,-1)\})+\bm{\varGamma_X}(\{(1,0),(-1,0)\}).
$$
\end{proof}

Another consequence of independence is the additivity, which is given in  Proposition \ref{indep2}. To prove it we first recall the following lemma.
\begin{lem}
	\label{int_equality}
	Let $(\bm{E,\varepsilon,m})$ be an arbitrary $\sigma$-finite measure space, and let $f_i: \bm{E}\rightarrow \mathbb{R},~ i=1,2,$ be two functions in $L^{\alpha}(\bm{E,\varepsilon,m})$.
	\begin{itemize}
		\item If $0<\alpha<1$, then either of the relations
		\begin{equation*}
		\int_{\bm{E}}|f_1(x)+f_2(x)|^{\alpha}\bm{m}(\ud x)=		\int_{\bm{E}}|f_1(x)|^{\alpha}\bm{m}(\ud x)+	\int_{\bm{E}}|f_2(x)|^{\alpha}\bm{m}(\ud x)
		\end{equation*}
		or
		\begin{equation*}
		\int_{\bm{E}}|f_1(x)-f_2(x)|^{\alpha}\bm{m}(\ud x)=		\int_{\bm{E}}|f_1(x)|^{\alpha}\bm{m}(\ud x)+	\int_{\bm{E}}|f_2(x)|^{\alpha}\bm{m}(\ud x)
		\end{equation*}
		implies 
		\begin{equation*}
		f_1(x)f_2(x)=0 \quad \bm{m}\mbox{\textendash} a.e..
		\end{equation*}
		\item If $1\leqslant\alpha\le2$, then
		\begin{eqnarray}\nonumber
		&&\int_{\bm{E}}|f_1(x)+f_2(x)|^{\alpha}\bm{m}(\ud x)=	\int_{\bm{E}}|f_1(x)-f_2(x)|^{\alpha}\bm{m}(\ud x)\\\nonumber	&&=\int_{\bm{E}}|f_1(x)|^{\alpha}\bm{m}(\ud x)+\int_{\bm{E}}|f_2(x)|^{\alpha}\bm{m}(\ud x)
		\end{eqnarray}
		implies
		\begin{equation*}
		f_1(x)f_2(x)=0 \quad \bm{m}\mbox{\textendash} a.e..
		\end{equation*}
	\end{itemize}
\end{lem}
\begin{proof}
When $\alpha=2$, the proof is trivial.
	For the case $\alpha\in(0,2)$, see Lemma 2.7.14 in \cite{Taqqu1994}. 
\end{proof}
\begin{prop}
\label{indep2}
	Let $\bm{X}=(X_1,X_2,X_3)$ be an S$\alpha $S random vector with $\alpha\in(0,2]$ and spectral measure $\bm{\varGamma_{X}}$. If $X_2$ and $X_3$ are independent, we have: for $\beta\ge0$ and $m=0$ or $1$,
	\[
	[X_1,X_2+X_3]_{\alpha,\beta,m}=[X_1,X_2]_{\alpha,\beta,m}+[X_1,X_3]_{\alpha,\beta,m}.
	\]
\end{prop}
\begin{proof}
	On the one hand, the characteristic function of the subset vector $(X_2, X_3)$ is: for all $\theta_1,\theta_2\in\mathbb R$,
	\begin{equation}\label{indep_proof_step_1}
	\mathbb E\exp\{i(0\cdot X_1+\theta_1X_2+\theta_2X_3)\}
	=\exp\left\{-\int_{S_3}|0\cdot s_1+\theta_1s_2+\theta_2s_3|^{\alpha}\bm{\varGamma_X}(\ud \bm s)\right\}.
	\end{equation}
	On the other hand, by the independence of $X_2$ and $X_3$, the above characteristic function can also be expressed as 
	\begin{eqnarray}\label{indep_proof_step_2}\nonumber
    &&\mathbb E\exp\{i(\theta_1X_2+\theta_2X_3)\}=\mathbb E\exp\{i\theta_1X_2\}\mathbb E\exp\{i\theta_2X_3\}\\\nonumber
    &&= \exp\left\{-\int_{S_3}\left|0\cdot s_1+\theta_1s_2+0\cdot s_3\right|^{\alpha}\bm{\varGamma_X}(\ud \bm s)\right\}\\
    &&\hspace{2cm}\times\exp\left\{-\int_{S_3}|(0\cdot s_1+0\cdot s_2+\theta_2s_3|^{\alpha}\bm{\varGamma_X}(\ud \bm s)\right\}.
	\end{eqnarray}
	It results from (\ref{indep_proof_step_1}) and (\ref{indep_proof_step_2}) that, for all $\theta_1,\theta_2\in\mathbb R$,
	\begin{equation}
    \label{indep_proof_step_3}
	\int_{S_3}|\theta_1s_2+\theta_2s_3|^{\alpha}\bm{\varGamma_X}(\ud \bm s)=|\theta_1|^{\alpha}\int_{S_3}|s_2|^{\alpha}\bm{\varGamma_X}(\ud \bm s)+|\theta_2|^{\alpha}\int_{S_3}|s_3|^{\alpha}\bm{\varGamma_X}(\ud \bm s).
	\end{equation}
	Replacing $\theta_2$ with $-\theta_2$ in (\ref{indep_proof_step_3}), we also have 
	\begin{equation}\label{indep_proof_step_4}
	\int_{S_3}|\theta_1s_2-\theta_2s_3|^{\alpha}\bm{\varGamma_X}(\ud \bm s)=|\theta_1|^{\alpha}\int_{S_3}|s_2|^{\alpha}\bm{\varGamma_X}(\ud \bm s)+|\theta_2|^{\alpha}\int_{S_3}|s_3|^{\alpha}\bm{\varGamma_X}(\ud \bm s).
	\end{equation}
	By Lemma \ref{int_equality} and the fact that (\ref{indep_proof_step_3}) and (\ref{indep_proof_step_4}) hold for all $\theta_1$ and $\theta_2$,  we obtain
    $$
    s_2s_3=0, ~ \bm{\varGamma_X}\mbox{ - } a.e..
    $$
This is equivalent to either
$$
s_2=0, ~ \bm{\varGamma_X}\mbox{ - } a.e.~\mbox{or}~s_3=0, ~ \bm{\varGamma_X}\mbox{ - } a.e..
$$
Therefore,
\small
    \begin{eqnarray*}
    &&[X_1,X_2+X_3]_{\alpha,\beta,m}\nonumber\\
    &&=\int_{\{(s_1,s_2,s_3)\in S_3:~|s_1|\le|s_2+s_3|\}} |s_1|^{\beta}|s_2+s_3|^{\alpha-\beta}\textnormal{sign}^{m}(s_1(s_2+s_3))\bm{\varGamma_X}(\ud \bm s)\\
&&\hspace{1cm}+\int_{\{(s_1,s_2,s_3)\in S_3:~|s_1|>|s_2+s_3|\}} |s_1|^{\alpha-\beta}|s_2+s_3|^{\beta}\textnormal{sign}^{m}(s_1(s_2+s_3))\bm{\varGamma_X}(\ud \bm s)\\
&&=\int_{\{(s_1,s_2,s_3)\in S_3:~|s_1|\le|s_3|,s_2=0\}} |s_1|^{\beta}|s_3|^{\alpha-\beta}\textnormal{sign}^{m}(s_1s_3)\bm{\varGamma_X}(\ud \bm s)\\&&\hspace{1cm}+\int_{\{(s_1,s_2,s_3)\in S_3:~|s_1|>|s_3|,s_2=0\}} |s_1|^{\alpha-\beta}|s_3|^{\beta}\textnormal{sign}^{m}(s_1s_3)\bm{\varGamma_X}(\ud \bm s)\\&&\hspace{1cm}+\int_{\{(s_1,s_2,s_3)\in S_3:~|s_1|\le|s_2|,s_3=0\}} |s_1|^{\beta}|s_2|^{\alpha-\beta}\textnormal{sign}^{m}(s_1s_2)\bm{\varGamma_X}(\ud \bm s)\\&&\hspace{1cm}+\int_{\{(s_1,s_2,s_3)\in S_3:~|s_1|>|s_2|,s_3=0\}} |s_1|^{\alpha-\beta}|s_2|^{\beta}\textnormal{sign}^{m}(s_1s_2)\bm{\varGamma_X}(\ud \bm s)\\
&&=\int_{\{(s_1,s_3)\in S_2:~|s_1|\le|s_3|\}} |s_1|^{\beta}|s_3|^{\alpha-\beta}\textnormal{sign}^{m}(s_1s_3)\bm{\varGamma_{(X_1,X_3)}}(\ud \bm s)\\&&\hspace{1cm}+\int_{\{(s_1,s_3)\in S_2:~|s_1|>|s_3|\}} |s_1|^{\alpha-\beta}|s_3|^{\beta}\textnormal{sign}^{m}(s_1s_3)\bm{\varGamma_{(X_1,X_3)}}(\ud \bm s)\\&&\hspace{1cm}+\int_{\{(s_1,s_2)\in S_2:~|s_1|\le|s_2|\}} |s_1|^{\beta}|s_2|^{\alpha-\beta}\textnormal{sign}^{m}(s_1s_2)\bm{\varGamma_{(X_1,X_2)}}(\ud \bm s)\\&&\hspace{1cm}+\int_{\{(s_1,s_2)\in S_2:~|s_1|>|s_2|\}} |s_1|^{\alpha-\beta}|s_2|^{\beta}\textnormal{sign}^{m}(s_1s_2)\bm{\varGamma_{(X_1,X_2)}}(\ud \bm s)\\
&&=[X_1,X_2]_{\alpha,\beta,m}+[X_1,X_3]_{\alpha,\beta,m},
\end{eqnarray*}
\normalsize
where $\bm{\varGamma_{(X_1,X_3)}}$ and $\bm{\varGamma_{(X_1,X_2)}}$ are two marginal spectral measures of $(X_1,X_2,$ $X_3)$.
\end{proof}

Proposition \ref{indep1} and Proposition \ref{indep2} provide two necessary conditions for the independence between two S$\alpha$S variables, in terms of the symmetric covariations. We now introduce a sufficient condition for the independence between $X_1$ and $X_2$.
\begin{prop}[Sufficient condition for independence]
	Let $\bm{X}=(X_1,X_2)$ be an S$\alpha $S random vector with $\alpha\in(0,2]$ and spectral measure $\bm{\varGamma_{X}}$.
If $[X_1, X_2]_{\alpha, \beta, 0}=0$ for some $\beta\ge0$, $X_1$ and $X_2$ are independent.
\end{prop}
\begin{proof}
Suppose $[X_1, X_2]_{\alpha, \beta, 0}=0$ for some $\beta\ge0$,  we have 
\begin{eqnarray}
&&[X_1, X_2]_{\alpha, \beta, 0}\nonumber\\&&=\int_{S_2}\left(  |s_1|^{\beta}|s_2|^{\alpha-\beta}\mathds 1_{\{|s_1|\le |s_2|\}}+|s_1|^{\alpha-\beta}|s_2|^{\beta}\mathds 1_{\{|s_1|>|s_2|\}}\right)\bm{\varGamma_X}(\ud \bm s)\nonumber\\&&=0.\label{indep_lem_1}
\end{eqnarray}
 Since the integrand in (\ref{indep_lem_1}) is non-negative, it yields 
 $$
 |s_1|^{\beta}|s_2|^{\alpha-\beta}\mathds 1_{\{|s_1|\le |s_2|\}}+|s_1|^{\alpha-\beta}|s_2|^{\beta}\mathds 1_{\{|s_1|>|s_2|\}}=0\quad  \bm{\varGamma_X}~\mbox{-}~a.e.,
 $$ which further implies
 $$s_1=0 \mbox{~if~} |s_1|\le |s_2| \mbox{~and~} s_2=0 \mbox{~if~} |s_1|> |s_2|\quad  \bm{\varGamma_X}~\mbox{-}~a.e..$$
Then we can write, for $\theta_1,\theta_2\in\mathbb R$,
\begin{eqnarray*}
&&\int_{S_2} |\theta_1s_1+\theta_2s_2|^\alpha\bm{\varGamma_X}(\ud \bm s)\\
&&=\int_{S_2}\left( |\theta_1s_1+\theta_2s_2|^\alpha\mathds 1_{\{|s_1|\le |s_2|\}}+ |\theta_1s_1+\theta_2s_2|^\alpha\mathds 1_{\{|s_1|> |s_2|\}}\right)\bm{\varGamma_X}(\ud \bm s)\\
&&=\int_{S_2} \left(|\theta_2s_2|^\alpha+ |\theta_1s_1|^\alpha\right)\bm{\varGamma_X}(\ud \bm s).
\end{eqnarray*}
Plugging the above equation into the characteristic function of $\bm{X}=(X_1,$ $X_2)$, we derive
$$
\mathbb{E}\exp\{i(\theta_1X_1+\theta_2X_2)\}=\mathbb{E}\exp\{i\theta_1 X_1\}\mathbb{E}\exp\{i\theta_2 X_2\},
$$
which leads to the fact that $X_1$ and $X_2$ are independent.
\end{proof}

Next, we present a sufficient condition for $X_1,X_2$ being James orthogonal. It is a straightforward consequence of the following lemma, which is proved in ``Appendix" (Section \ref{subsection:proof_of_general_ineq}).
\begin{lem}
\label{general_ineq}
If $[X_1, X_2]_{\alpha, k, 1}=0$ for all $k\in\mathbb{Z}$,
\begin{equation}
\label{lower_bound}
\|X_1+X_2\|_\alpha\ge\min\left\{2^{1-1/\alpha},1\right\}\max\left\{\|X_1\|_\alpha,\|X_2\|_\alpha\right\}.
\end{equation}
\end{lem}

Now we state the sufficient condition for $X_1,X_2$ being James orthogonal to each other, using the result of Lemma \ref{general_ineq}.
\begin{prop}[Sufficient condition for James orthogonality]
\label{James_inequality}
Let $(X_1,X_2)$ be a jointly S$\alpha$S random vector with $\alpha\in[1,2]$. If $[\lambda X_1,X_2]_{\alpha,k,1}=0$ for all $\lambda\in\mathbb R$ and all $k\in\mathbb Z$, $X_1,X_2$ are James orthogonal to each other.
\end{prop}
\begin{proof}
On the one hand, since $[\lambda X_1,X_2]_{\alpha,k,1}=0$ for all $\lambda\in\mathbb R$ and all $k\in\mathbb Z$, it follows from Lemma \ref{general_ineq} and the fact that $\alpha\ge1$ that, for all $\lambda\in\mathbb R$,
\begin{eqnarray}
\label{James1}
\|\lambda X_1+X_2\|_\alpha&\ge&\min\left\{2^{1-1/\alpha},1\right\}\max\left\{\|\lambda X_1\|_\alpha,\|X_2\|_\alpha\right\}\nonumber\\
&=&\max\left\{\|\lambda X_1\|_\alpha,\|X_2\|_\alpha\right\}\nonumber\\
&\ge&\|X_2\|_\alpha,
\end{eqnarray}
i.e., $X_1$ is James orthogonal to $X_2$.

On the other hand, the fact that (see Corollary \ref{non-linear2}), for all $\lambda\neq0$, $k\in\mathbb Z$,
$$
[\lambda X_1,X_2]_{\alpha,k,1}=|\lambda|^\alpha\left[X_1,\lambda^{-1}X_2\right]_{\alpha,k,1}=|\lambda|^\alpha\left[\lambda^{-1}X_2,X_1\right]_{\alpha,k,1}
$$
implies $[\lambda X_2,X_1]_{\alpha,k,1}=0$ for all $\lambda\in\mathbb R$, all $k\in\mathbb Z$. Then similar to (\ref{James1}) we can show that $X_2$ is also James orthogonal to $X_1$. Hence Proposition \ref{James_inequality} is proved.
\end{proof}

We remark that $[X_1,X_2]_{\alpha,k,1}=0$ for all $k\in\mathbb Z$ is not sufficient to let $X_1$ be James orthogonal to $X_2$, because it generally does not imply $[\lambda X_1,X_2]_{\alpha,\beta,1}=0$ for all $\lambda\in\mathbb R$. However, the condition in Proposition \ref{James_inequality} is strong enough, since it leads to the fact that $X_1$ and $X_2$ are James orthogonal to each other. Recall that the zero conventional covariation, i.e., $[X_1,X_2]_\alpha=0$ only implies that $X_1$ is James orthogonal to $X_2$, due to the fact that it is not symmetric.
\section{Conclusion and Future Research}
\label{conclusion}
In this paper, we have proposed the notion of symmetric covariations to measure the dependence of the S$\alpha$S random variables. Evidently, these symmetric covariations preserve  properties similar to those of the covariance structure in the Gaussian case. When compared to covariation, the main advantages of  symmetric covariation are that it is well defined for all $\alpha\in(0,2]$, is symmetric, and can be used to extend the Taylor series of Gaussian characteristic functions. Other desirable covariance-like properties of symmetric covariation that have been obtained are the Cauchy-Schwarz inequality, independence, and orthogonality. With  symmetric covariation, we can establish a convergent series representation of the characteristic function of the jointly S$\alpha$S distribution and measure the dependence between the S$\alpha$S variables.  The series representation of the jointly S$\alpha$S distribution will have significant applications in statistics; examples include the approximation of jointly S$\alpha$S distributions, simulation of stable random vectors, and estimation of the distribution parameters. Constructing estimators of symmetric covariation will contribute new and efficient approaches to the approximation problems of jointly S$\alpha$S distributions.

\section*{Acknowledgements}
The authors would like to thank Professor Vygantas Paulauskas for very stimulating communications on measuring dependence between S$\alpha$S variables. The authors also thank the referee and the editor for their comments on the manuscript which lead to many improvements of the presentation of this paper.

\appendix
\section{Proofs of Statements}
\label{section:appendix}
\subsection{Proof of Lemma \ref{poly}}
\label{subsection:proof_of_poly}
\begin{proof}
For $\beta\in\mathbb R_+\backslash\mathbb Z_+$, from (\ref{def:derivative}) we write 
\begin{equation}
\label{def:derivative_step1}
{}_{a}^{m}\text{D}_x^\beta \left(|x-a|^{p}\right)=\mathds 1_{\{x\ge a\}}{}_a\widetilde {\text{D}}_x^\beta \left(|x-a|^{p}\right)+(-1)^{m}\mathds 1_{\{x<a\}}{}_x\widetilde {\text{D}}_a^\beta \left(|x-a|^{p}\right).
\end{equation}
Let $n=\lfloor\beta\rfloor+1$. On the one hand, if $x\ge a$, taking the left Riemann-Liouville fractional derivative of the function $x\mapsto |x-a|^{p}$, it yields
\begin{eqnarray}
\label{left_derivative_comp}
&&{}_a\widetilde {\text{D}}_x^\beta \left(|x-a|^{p}\right)=\frac{1}{\Gamma(n-\beta)}\frac{\ud^n}{\ud x^n}\int_a^x(t-a)^p(x-t)^{n-\beta+1}\ud t\nonumber\\
&&=\frac{1}{\Gamma(n-\beta)}\frac{\ud^n}{\ud x^n}\int_0^1(x-a)^p(1-\tau)^p(x-a)^{n-\beta-1}\tau^{n-\beta-1}(x-a)\ud\tau\nonumber\\
&&=\frac{1}{\Gamma(n-\beta)}\left(\frac{\ud^n}{\ud x^n}(x-a)^{p+n-\beta}\right)\left(\int_0^1(1-\tau)^p\tau^{n-\beta-1}\ud\tau\right)\nonumber\\
&&=\frac{B(n-\beta,p+1)}{\Gamma(n-\beta)}\left(\frac{\ud^n}{\ud x^n}(x-a)^{p+n-\beta}\right)\nonumber\\
 &&=\frac{\Gamma(p+1)}{\Gamma(n-\beta+p+1)}\left(\frac{\Gamma(n-\beta+p+1)}{\Gamma(p-\beta+1)}(x-a)^{p-\beta}\right)\nonumber\\
&&=\frac{\Gamma(p+1)}{\Gamma(p-\beta+1)}(x-a)^{p-\beta}.
\end{eqnarray}
Here we have used the change of variables $t=x-(x-a)\tau$ and the fact that the beta function $B$ satisfies
$$
B(x,y)=\frac{\Gamma(x)\Gamma(y)}{\Gamma(x+y)},~\mbox{for all}~x,y>0.
$$
On the other hand, if $x<a$, the right Riemann-Liouville fractional derivative yields
\begin{eqnarray}
\label{right_derivative_comp}
&&{}_x\widetilde {\text{D}}_a^\beta \left(|x-a|^{p}\right)=(-1)^n\frac{1}{\Gamma(n-\beta)}\frac{\ud^n}{\ud x^n}\int_x^a(a-t)^p(t-x)^{n-\beta+1}\ud t\nonumber\\
&&=(-1)^n\frac{1}{\Gamma(n-\beta)}\frac{\ud^n}{\ud x^n}\int_0^1(a-x)^p(1-\tau)^p(a-x)^{n-\beta-1}\tau^{n-\beta-1}(a-x)\ud\tau\nonumber\\
&&=\frac{(-1)^n}{\Gamma(n-\beta)}\left(\frac{\ud^n}{\ud x^n}(a-x)^{p+n-\beta}\right)\left(\int_0^1(1-\tau)^p\tau^{n-\beta-1}\ud\tau\right)\nonumber\\
&&=\frac{(-1)^nB(n-\beta,p+1)}{\Gamma(n-\beta)}\left(\frac{\ud^n}{\ud x^n}(a-x)^{p+n-\beta}\right)\nonumber\\
&&=\frac{(-1)^n\Gamma(p+1)}{\Gamma(n-\beta+p+1)}\left((-1)^n\frac{\Gamma(n-\beta+p+1)}{\Gamma(p-\beta+1)}(a-x)^{p-\beta}\right)\nonumber\\
&&=\frac{\Gamma(p+1)}{\Gamma(p-\beta+1)}(a-x)^{p-\beta}.
\end{eqnarray}
Here we have taken $t=x+(a-x)\tau$. Therefore (\ref{dev:poly}) holds for $\beta$ being non-integer, by using (\ref{def:derivative_step1}), (\ref{left_derivative_comp}) and (\ref{right_derivative_comp}).

For $\beta\in\mathbb{Z}_+$, by using $(\ref{integer_case})$ we write
\begin{eqnarray}
{}_{a}^{m}\text{D}_x^\beta \left(|x-a|^{p}\right)
&=&\textnormal{sign}^{m+n+1}(x-a)\frac{\ud^\beta }{\ud x^\beta}\left(|x-a|^{p}\right)\nonumber\\&=&\frac{\Gamma(p+1)}{\Gamma(p-\beta+1)}|x-a|^{p-\beta}\textnormal{sign}^{m}(x-a)\nonumber,
\end{eqnarray}
where $n=\lfloor \beta\rfloor+1$. Hence (\ref{dev:poly}) is obtained for all real numbers $\beta\ge0$.
\end{proof}
\subsection{Proof of Lemma \ref{96}}
\label{subsection:proof_of_96}
\begin{proof}
	On the one hand, in view of Lemma 2.7.5 in \cite{Taqqu1994}, $\bm{Y}=(Y_1,Y_2)$ is also an S$\alpha $S random vector. 
	On the other hand, considering $\bm{Y}=\left(\sum_{k=1}^{n}a_kX_k,\right.$ $\left.\sum_{k=1}^{n}b_kX_k\right)$, we can rewrite the characteristic function as:
	\begin{eqnarray*}
% 	\label{98}
	&&\mathbb E\exp\{i(\theta_1Y_1+\theta_2Y_2)\}=\mathbb E\exp\left\{i\sum_{k=1}^{n}(\theta_1a_k+\theta_2b_k)X_k\right\}\nonumber\\
	&&=\exp\left\{-\int_{S_n}\left|\theta_1\sum_{k=1}^{n}a_ks_k+\theta_2\sum_{k=1}^{n}b_ks_k\right|^{\alpha}\bm{\varGamma_X}(\ud \bm s)\right\}.
	\end{eqnarray*}
    Now we want to show 
    \begin{equation}
    \label{99}
  \int_{S_n}\left|\theta_1\sum_{k=1}^{n}a_ks_k+\theta_2\sum_{k=1}^{n}b_ks_k\right|^{\alpha}\bm{\varGamma_X}(\ud \bm s)=\int_{S_2}|\theta_1t_1+\theta_2t_2|^{\alpha}\bm{\varGamma}(\ud \bm t),
\end{equation}
with $\bm{\varGamma}=\widehat{\bm{\varGamma_{X}}}\circ h^{-1}$.
	To verify the above equation, we first write
	\begin{eqnarray*}
	\int_{S_2}|\theta_1t_1+\theta_2t_2|^{\alpha}\bm{\varGamma}(\ud \bm t)&=&\int_{S_2}|\theta_1t_1+\theta_2t_2|^{\alpha}\widehat{\bm{\varGamma_{X}}}(\ud(h^{-1}(\bm t))).
	\end{eqnarray*}
	Then applying Lemma $\ref{measure}$, we have
	\begin{eqnarray*}
	&&\int_{S_2}|\theta_1t_1+\theta_2t_2|^{\alpha}\widehat{\bm{\varGamma_{X}}}(\ud(h^{-1}(\bm t)))\nonumber\\
    &&=\int_{S_n}\left|\frac{\theta_1\sum_{k=1}^{n}a_ks_k+\theta_2\sum_{k=1}^nb_ks_k}{\big(\big(\sum_{k=1}^{n}a_ks_k\big)^{2}+\big(\sum_{k=1}^{n}b_ks_k\big)^{2}\big)^{1/2}}\right|^{\alpha}\widehat{\bm{\varGamma_{X}}}(\ud \bm s)\nonumber\\
&&=\int_{S_n}\left|\theta_1\sum_{k=1}^{n}a_ks_k+\theta_2\sum_{k=1}^{n}b_ks_k\right|^{\alpha}\bm{\varGamma_X}(\ud \bm s).
	\end{eqnarray*}
	Therefore, $(\ref{99})$ holds and hence $\bm{\varGamma}=\widehat{\bm{\varGamma_{X}}}\circ h^{-1}$, denoted by $\bm{\varGamma_Y}$, is a spectral measure of $Y$.
\end{proof}
\subsection{Proof of Proposition \ref{prop:equiv}}
\label{subsection:proof_of_prop}
\begin{proof}
We first prove $(i)$. By using Remark \ref{rmk3}, we have
\begin{eqnarray*}
&&[aX_1,bX_2]_{\alpha,1,1}+[aX_1,bX_2]_{\alpha,\alpha-1,1}\\&&=\int_{\{(s_1,s_2)\in S_2:~|as_1|\le|bs_2|\}} \left(as_1\right)^{\langle1\rangle}\left(bs_2\right)^{\langle\alpha-1\rangle}\bm{\varGamma_X}(\ud \bm s)\\
&&\hspace{3cm}+\int_{\{(s_1,s_2)\in S_2:~|as_1|>|bs_2|\}} \left(as_1\right)^{\langle\alpha-1\rangle}\left(bs_2\right)^{\langle1\rangle}\bm{\varGamma_X}(\ud \bm s)\\
&&+\int_{\{(s_1,s_2)\in S_2:~|as_1|\le|bs_2|\}} \left(as_1\right)^{\langle\alpha-1\rangle}\left(bs_2\right)^{\langle1\rangle}\bm{\varGamma_X}(\ud \bm s)\\
&&\hspace{3cm}+\int_{\{(s_1,s_2)\in S_2:~|as_1|>|bs_2|\}} \left(as_1\right)^{\langle1\rangle}\left(bs_2\right)^{\langle\alpha-1\rangle}\bm{\varGamma_X}(\ud \bm s)\\
&&=\int_{S_2} \left(as_1\right)^{\langle1\rangle}\left(bs_2\right)^{\langle\alpha-1\rangle}\bm{\varGamma_X}(\ud \bm s)+\int_{S_2} \left(as_1\right)^{\langle\alpha-1\rangle}\left(bs_2\right)^{\langle1\rangle}\bm{\varGamma_X}(\ud \bm s)\\
&&=ab^{\langle\alpha-1\rangle}[X_1,X_2]_{\alpha}+ba^{\langle\alpha-1\rangle}[X_2,X_1]_{\alpha}.
\end{eqnarray*}
Here in the last equality we have used the fact that $s_1^{\langle1\rangle}=s_1$ for all $s_1\in\mathbb R$. Hence (\ref{result_1}) is proved.\\ Next we prove $(ii)$.
Using (\ref{result_1}), the following system of linear equations hold:
\begin{equation}
\label{result_2'}
\left\{\hspace{-0.5cm}\begin{array}{ll}
     &a_1b_1^{\langle\alpha-1\rangle}[X_1,X_2]_{\alpha}+b_1a_1^{\langle\alpha-1\rangle}[X_2,X_1]_{\alpha}\\&\hspace{1cm}=[a_1X_1,b_1X_2]_{\alpha,1,1}+[a_1X_1,b_1X_2]_{\alpha,\alpha-1,1}, \\
     & a_2b_2^{\langle\alpha-1\rangle}[X_1,X_2]_{\alpha}+b_2a_2^{\langle\alpha-1\rangle}[X_2,X_1]_{\alpha}\\&\hspace{1cm}=[a_2X_1,b_2X_2]_{\alpha,1,1}+[a_2X_1,b_2X_2]_{\alpha,\alpha-1,1}.
\end{array}\right.
\end{equation}
Since $a_1b_1^{\langle\alpha-1\rangle}b_2a_2^{\langle\alpha-1\rangle}\ne a_2b_2^{\langle\alpha-1\rangle}b_1a_1^{\langle\alpha-1\rangle}$, the solution of (\ref{result_2'}) is uniquely obtained as in
 (\ref{result_2}). 
\end{proof}
\subsection{Proof of Lemma \ref{general_ineq}}
\label{subsection:proof_of_general_ineq}
\begin{proof}
Assume $[X_1, X_2]_{\alpha, k, 1}=0$ for all $k\in\mathbb Z$.  Applying this assumption to Theorem \ref{lem:scale}, we have 
\begin{equation}
\label{sigma_even}
\sigma^\alpha_{\bm X}(1, 1)=\sum_{k \mbox{ is even}}\frac{(\alpha)_k}{k!}[X_1,  X_2]_{\alpha, k, 0}=\widetilde{T}_1+\widetilde{T}_2,
\end{equation}
where 
\begin{equation}
\label{widetilde1}
\widetilde{T}_1=\sum_{k\mbox{ is even}}\frac{(\alpha)_k}{k!}\int_{\{(s_1,s_2)\in S_2:~|s_1|\le|s_2|\}}|s_1|^k|s_2|^{\alpha-k}\bm{\varGamma_X}(\ud \bm s)
\end{equation}
and
\begin{equation}
\label{widetilde2}
\widetilde{T}_2=\sum_{k\mbox{ is even}}\frac{(\alpha)_k}{k!}\int_{\{(s_1,s_2)\in S_2:~|s_1|>|s_2|\}}|\lambda s_1|^{\alpha-k}|s_2|^k\bm{\varGamma_X}(\ud \bm s).
\end{equation}
From Proposition \ref{rmk:taylor_abs}, we have that for all $b\ne0$ and all $x\in[-|b|,|b|]$,
\begin{eqnarray}
\label{sum}
	&&|x+b|^\alpha+|x-b|^\alpha\nonumber\\&=&\sum_{k=0}^{+\infty}\frac{(a)_k}{k!}|b|^{\alpha-k}\textnormal{sign}^k(b)x^k+\sum_{k=0}^{+\infty}\frac{(a)_k}{k!}|b|^{\alpha-k}\textnormal{sign}^k(-b)x^k\nonumber\\
	&=&\sum_{k=0}^{+\infty}\frac{(a)_k}{k!}|b|^{\alpha-k}\textnormal{sign}^k(b)x^k\left(1+(-1)^k\right)\nonumber\\
	&=&2\sum_{k\mbox{ is even}}\frac{(a)_k}{k!}|b|^{\alpha-k}\textnormal{sign}^k(b)x^k\nonumber\\
    &=&2\sum_{k\mbox{ is even}}\frac{(a)_k}{k!}|x|^k|b|^{\alpha-k}.
\end{eqnarray}
It follows from (\ref{sum}), (\ref{widetilde1}) and (\ref{widetilde2}) that 
$$
\widetilde{T}_1=\frac{1}{2}\int_{\{(s_1,s_2)\in S_2:~|s_1|\le|s_2|\}}\left(|s_1+s_2|^\alpha+|s_1-s_2|^{\alpha}\right)\bm{\varGamma_X}(\ud \bm s)
$$
and 
$$
\widetilde{T}_2=\frac{1}{2}\int_{\{(s_1,s_2)\in S_2:~|s_1|>|s_2|\}}\left(|s_1+s_2|^\alpha+|s_1-s_2|^{\alpha}\right)\bm{\varGamma_X}(\ud \bm s).
$$
(\ref{sigma_even}) then becomes 
\begin{equation}
\label{sigma_2}
\int_{ S_2}|s_1+s_2|^\alpha\bm{\varGamma_X}(\ud \bm s)=\frac{1}{2}\int_{ S_2}\left(|s_1+s_2|^\alpha+|s_1-s_2|^{\alpha}\right)\bm{\varGamma_X}(\ud \bm s).
\end{equation}
Now recall the following inequality (see, e.g., Lemma 2.7.13 in \cite{Taqqu1994}): for $x,y\in\mathbb R$ and $p\ge0$,
\begin{equation}
\label{ineq_alg}
|x+y|^p+|x-y|^p\ge\min\{2^{p},2\}\max\{|x|^p,|y|^p\}.
\end{equation}
It then results from (\ref{sigma_2}) and (\ref{ineq_alg}) that
\small
$$
\int_{ S_2}|s_1+s_2|^\alpha\bm{\varGamma_X}(\ud \bm s)\ge\min\{2^{\alpha-1},1\}\max\left\{\int_{ S_2}|s_1|^{\alpha}\bm{\varGamma_X}(\ud \bm s), \int_{ S_2}|s_2|^{\alpha}\bm{\varGamma_X}(\ud \bm s)\right\},
$$
\normalsize
i.e.,
$$
\|X_1+X_2\|_\alpha^\alpha\ge\min\left\{2^{\alpha-1},1\right\}\max\left\{\|X_1\|_\alpha^\alpha,\|X_2\|_\alpha^\alpha\right\}.
$$
This proves Lemma \ref{general_ineq}.
\end{proof}
\bibliographystyle{abbrvnat}
\bibliography{ml}
\end{document}